\documentclass{amsart}

\usepackage[utf8]{inputenc}
\usepackage[T1]{fontenc}
\usepackage{lmodern}
\usepackage[english]{babel}
\usepackage{fullpage}
\usepackage{amsmath}
\usepackage{amsthm}
\usepackage{amsfonts}
\usepackage{amssymb}
\usepackage{hyperref}
\usepackage{color}
\usepackage{verbatim}

\newcommand{\R}{\mathbb{R}}
\newcommand{\C}{\mathbb{C}}

\newcommand{\Z}{\mathbb{Z}}
\newcommand{\Ri}{\mathcal{R}}

\renewcommand{\H}{\mathcal{H}}

\newcommand{\supp}{\operatorname{supp}}
\newcommand{\diam}{\operatorname{diam}}
\newcommand{\dist}{\operatorname{dist}}

\newcommand{\Stop}{\mathsf{Stop}}
\newcommand{\Top}{\mathsf{Top}}
\newcommand{\Tree}{\mathsf{Tree}}
\newcommand{\Good}{\mathsf{Good}}
\newcommand{\Reg}{\mathsf{Reg}}
\newcommand{\ND}{\mathsf{ND}}

\newcommand{\D}{\mathcal{D}}

\renewcommand{\l}{\ell}
\newcommand{\dmu}{\delta_{\mu}}
\newcommand{\ntmu}{||T\mu||_{L^2(\mu)}}
\renewcommand{\epsilon}{\varepsilon}
\newcommand{\refeq}[1]{{(\ref{#1})}}
\newcommand{\I}{\mathsf{I}}

\newtheorem{theorem}{Theorem}
\newtheorem{lemma}{Lemma}
\newtheorem{corollary}{Corollary}

\newtheorem{other}{\bf Theorem}

\newtheorem{otherl}{\bf Lemma}

\begin{document}

\title[Short title]{Geometric conditions for the $L^2$-boundedness of singular integral operators with odd kernels with respect to measures with polynomial growth in $\R^d$}


\author{Daniel Girela-Sarri\'{o}n}
\address{Departament de Matem\`{a}tiques\\
Edifici C Facultat de Ci\`{e}ncies\\
08193 Bellaterra (Barcelona, Spain)}
\email{dgirela@mat.uab.cat}
\thanks{I wish to thank my advisor, X. Tolsa, for his valuable guidance during the preparation of this work. I was supported by project MTM-2010-16232 (MICINN, Spain) and partially supported by the ERC grant 320501 of the European Research Council (FP7/2007-2013), and projects MTM-2013-44304-P (MICINN, Spain) and 2014-SGR-75 (Catalonia).}



\date{\today}

\dedicatory{}

\begin{abstract}
Let $\mu$ be a finite Radon measure in $\R^d$ with polynomial growth of degree $n$, although not necessarily $n$-AD regular. We prove that under some geometric conditions on $\mu$ that are closely related to rectifiability and involve the so-called $\beta$-numbers of Jones, David and Semmes, all singular integral operators with an odd and sufficiently smooth Calder\'{o}n-Zygmund kernel are bounded in $L^2(\mu)$. As a corollary, we obtain a lower bound for the Lipschitz harmonic capacity of a compact set in $\R^d$ only in terms of its metric and geometric properties.
\end{abstract}

\maketitle

\section{Introduction}

We say that a function $k\colon \R^d\times\R^d\setminus\{(x,y)\in\R^d\times\R^d\colon x=y\}$ is an $n$-dimensional Calder\'{o}n-Zygmund kernel if there are constants $c>0$ and $0<\delta\leq 1$ such that
\begin{equation*}
|k(x,y)|\leq \frac{c}{|x-y|^n}\;\text{ if }\;x\neq y
\end{equation*}
and
\begin{equation}
\label{eq:defCZKernel}
|k(x,y)-k(x',y)|+|k(y,x)-k(y,x')|\leq c\frac{|x-x'|^\delta}{|x-y|^{n+\delta}}\;\text{ if }\;|x-x'|\leq\frac{|x-y|}{2}.
\end{equation}
\par\medskip
Given a signed Radon measure $\nu$ in $\R^d$ and $x\in\R^d$, we define
\begin{equation*}
T\nu(x) = \int k(x,y)d\nu(y), \;\; x\in \R^d\setminus\supp(\nu)
\end{equation*}
and we say that $T$ is a singular integral operator with kernel $k$. The integral above need not be convergent for $x\in\supp(\nu)$, and this is why one introduces the truncated operators associated to $T$, which are defined, for every $\epsilon>0$, by
\begin{equation*}
T_\epsilon\nu(x) = \int_{|x-y|>\epsilon}k(x,y)d\nu(y), \; x\in\R^d.
\end{equation*}
Notice that the integral above is absolutely convergent if, for example, $|\nu|(\R^d)<\infty$.
\par\medskip
If $\mu$ is a fixed positive Radon measure in $\R^d$ and $f\in L^1_{loc}(\mu)$, we set
\begin{equation*}
T_\mu f(x) = T(f\mu)(x),\; x\in\R^d\setminus\supp(\mu)
\end{equation*}
and, for $\epsilon>0$,
\begin{equation*}
T_{\mu,\epsilon}f(x) = T_\epsilon(f\mu)(x), \; x\in\R^d.
\end{equation*}
We say that $T_\mu$ is bounded in $L^2(\mu)$ if there is a constant $C>0$ such that, for all $\epsilon>0$, $||T_{\mu,\epsilon}f||_{L^2(\mu)}\leq C||f||_{L^2(\mu)}$ for all $f$ in $L^2(\mu)$. The norm of $T_\mu$ is the infimum of all those constants $C$ (an analogous definition is used to define the boundedness of $T_\mu$ in other spaces). Probably, the most important examples of this class of operators are the $n$-dimensional Riesz transform, given by
\begin{equation*}
\Ri\nu(x) = \int \frac{x-y}{|x-y|^{n+1}}d\nu(y)
\end{equation*}
and its one-dimensional analog in $\R^2\equiv\C$, the Cauchy transform, defined by
\begin{equation*}
\mathcal{C}\nu(z) = \int\frac{d\nu(\zeta)}{\zeta-z}.
\end{equation*}
\par\medskip
In this paper, we study $L^2(\mu)$-boundedness of Calder\'{o}n-Zygmund operators with sufficiently smooth convolution-type kernels. More precisely, we will consider kernels of the form $k(x,y)=K(x-y)$, where $K\colon\R^d\setminus\{0\}\rightarrow\R$ is an odd and $\mathcal{C}^2$ function that satisfies
\begin{equation*}
|\nabla^{j} K(x)|\leq \frac{C(j)}{|x|^{n+j}}\text{ for all }x\neq 0\text{ and }j\in\{0,1,2\}.
\end{equation*}
It is easy to check that the inequalities above imply that $k$ is a Calderón-Zygmund kernel with $\delta=1$ in \refeq{eq:defCZKernel}. We will denote by $\mathcal{K}^n(\R^d)$ the class of all these kernels.
\par\medskip
In \cite{Tolsa-pubMat}, Tolsa proved the following result\footnote{Tolsa's result in \cite{Tolsa-pubMat} is actually stated for operators with smoother kernels than the ones we consider here. However, after the publication of \cite{Tolsa-Alphas}, it is obvious that it can be generalized to obtain Theorem \ref{Thm:ThmA}}:
\par\medskip
\begin{other}
\label{Thm:ThmA}
Let $\mu$ be a Radon measure in $\C$ without atoms. If the Cauchy transform $\mathcal{C}_\mu$ is bounded in $L^2(\mu)$, then all $1$-dimensional Calder\'{o}n-Zygmund operators $T_\mu$ with kernels in $\mathcal{K}^1(\C)$ are also bounded in $L^2(\mu)$.
\end{other}
\par\medskip
In order to prove this result, Tolsa relied on a suitable corona decomposition for measures with linear growth and finite curvature\footnote{We will not enter into details about curvature of measures and its relationship with the boundedness of the Cauchy transform here, but an interested reader is encouraged to read \cite[Chapters 3 and 7]{Tolsa-Libro} for further information on this issue.} and split the operator $T$ into a sum of different operators $K_R$, each of which are associated to a \emph{tree} of the corona decomposition. The operators $K_R$ are bounded because on each tree the measure $\mu$ can be approximated by arc length on an Ahlfors-David regular curve and, moreover, the operators $K_R$ behave in a quasiorthogonal way. However, as that corona construction relied heavily on the relationship between the Cauchy transform and curvatures of measures, it could not be easily generalized to higher dimensions. Nevertheless, using a new corona decomposition that involves the $\beta$-numbers of Jones, David and Semmes instead of curvature and is valid for all dimensions, Azzam and Tolsa \cite{Azzam-Tolsa} have recently proved the following:
\par\medskip
\begin{other}
Let $\mu$ be a finite Radon measure with compact support in $\C$ with linear growth. Then, for all $\epsilon>0$,
\begin{equation*}
||C_\epsilon\mu||^2_{L^2(\mu)} \lesssim ||\mu|| + \iint_0^{\infty}\beta^n_{\mu,2}(x,r)^2\theta_\mu^1(x,r)\frac{dr}{r}d\mu(x).
\end{equation*}
\end{other}

\par\medskip
Some notions need to be defined here: first of all, a Borel measure $\mu$ in $\R^d$ is said to have polynomial growth of degree $n$ if there is a constant $c_0\geq 0$ such that $\mu[B(x,r)]\leq c_0r^n$ for all $x\in \R^d$ and all $r>0$ (when $n=1$, $\mu$ is said to have linear growth). Secondly, given a ball $B(x,r)\subset\R^d$, we define
\begin{equation*}
\theta_\mu^n(x,r) = \frac{\mu(B(x,r))}{r^n}.
\end{equation*}
Finally, for $1\leq p<\infty$, the $\beta^n_{\mu,p}$ coefficient of a ball $B$ with radius $r(B)$ is defined by
\begin{equation*}
\beta^n_{\mu,p}(B) = \inf_L\left(\frac{1}{r(B)^n}\int_B\left(\frac{\dist(y,L)}{r(B)}\right)^pd\mu(y)\right)^{\frac{1}{p}},
\end{equation*}
where the infimum is taken over all $n$-planes $L\subset \R^d$. It is worth mentioning that these $\beta^n_{\mu,p}$ coefficients are a generalization of the $\beta$ numbers introduced by Jones in \cite{Jones}, where he used them to characterized compact subsets of the plane that are contained in a rectifiable set. Furthermore, David and Semmes proved in \cite{David-Semmes-Asterisque} that an $n$-AD-regular measure $\mu$ is uniformly rectifiable if, and only if, there is some constant $c>0$ such that, for every ball $B$ with centre on $\supp(\mu)$,
\begin{equation}
\label{eq:UnifRect}
\int_B\int_0^{r(B)}\beta^n_{\mu,2}(x,r)^2\frac{dr}{r}d\mu(x)\leq c\mu(B).
\end{equation}
\par\medskip
Very recently, Azzam and Tolsa (see \cite{Azzam-Tolsa} and \cite{Tolsa-PartI}) have shown that a positive and finite Borel measure $\mu$ in $\R^d$ with
\begin{equation*}
0<\limsup_{r\rightarrow 0}\,\theta^n_\mu(x,r)<\infty \text{ for }\mu-\text{a.e. }x\in\R^d
\end{equation*} 
is $n$-rectifiable if, and only if,
\begin{equation}
\label{eq:Rectif}
\int_0^1\beta_{\mu,2}(x,r)^2\frac{dr}{r}<\infty
\end{equation}
for $\mu$-a.e. $x\in\R^d$.
\par\medskip
Using the corona decomposition from \cite{Azzam-Tolsa}, we prove the following result:
\begin{theorem}
\label{Th:MainThm}
Let $\mu$ be a finite Radon measure in $\R^d$ with polynomial growth of degree $n$ and such that, for all balls $B\subset\R^d$ with radius $r(B)$,
\begin{equation}
\label{eq:MainThm}
\int_B \int_0^{r(B)}\beta^n_{\mu,2}(x,r)^2\theta^n_\mu(x,r)\frac{dr}{r}d\mu(x) \lesssim \mu(B).
\end{equation}
Then, all Calder\'{o}n-Zygmund operators $T_\mu$ with kernels in $\mathcal{K}^n(\R^d)$ are bounded in $L^2(\mu)$.
\end{theorem}
Notice that \refeq{eq:MainThm} is a quantitative version of \refeq{eq:Rectif}, just like \refeq{eq:UnifRect}, with no assumptions on the AD-regularity of $\mu$. A trivial example of a measure $\mu$ that is not $n$-AD-regular and satisfies \refeq{eq:MainThm} is the area measure on a square (with $d=2$ and $n=1$). Of course, the most interesting examples with regard to this result will arise from measures that have \emph{some $n$-dimensional nature} (e.g., measures supported on sets with Hausdorff dimension equal to $n$).
\par\medskip
When $n=d-1$, the previous result can be applied to get an interesting estimate for the Lipschitz harmonic capacity. Recall that the Lipschitz harmonic capacity of a compact set $E\subset\R^d$ is defined by
\begin{equation*}
\kappa(E) = \sup|\langle\Delta\varphi, 1\rangle|,
\end{equation*}
where the supremum is taken over all Lipschitz functions $\varphi\colon\R^d\rightarrow\R$ that are harmonic in $\R^{d}\setminus E$ with $||\nabla\varphi||_{\infty}\leq 1$. Here $\langle\Delta\varphi,1\rangle$ denotes the action of the compactly supported distributional Laplacian $\Delta\varphi$ on the function 1. This notion was introduced by Paramonov \cite{Paramonov} to study the problem of $C^1$ harmonic approximation on compact subsets of $\R^d$ and, as it was proved by Mattila and Paramonov in \cite{Mattila-Paramonov}, serves to characterize removable sets for Lipschitz harmonic functions as those sets $E$ with $\kappa(E)=0$. From Theorem \ref{Th:MainThm}, we obtain the following:

\begin{corollary}
\label{Cor:LipHarmCap}
Let $E$ be a compact set in $\R^{n+1}$. Then,
\begin{equation}
\label{Eq:LipHarmCap}
\kappa(E) \gtrsim \sup \mu(E),
\end{equation}
where the supremum is taken over all positive Borel measures $\mu$ supported on $E$ such that
\begin{equation}
\label{eq:SUPKAPPA}
\sup_{x\in\R^{n+1}, R>0}\left\{\theta^n_\mu(x,R)+\int_0^\infty\beta_{\mu,2}(x,r)^2\theta_\mu^n(x,r)\frac{dr}{r}\right\}\leq 1.
\end{equation}
\end{corollary}

\par\medskip
A very interesting problem would be to show that, in fact, $\gtrsim$ may be substituted by $\approx$ in \refeq{Eq:LipHarmCap}, as an analog to the comparabilty between the analytic capacity $\gamma$ and the capacity $\gamma_+$ obtained by Tolsa in \cite{To:Acta}. This would serve to characterize removable sets for Lipschitz harmonic functions in a metric-geometric way and also to prove the bi-Lipschitz invariance of Lipschitz harmonic capacity, which is still unknown. Indeed, whenever a measure $\mu$ satisfies \refeq{eq:SUPKAPPA}, it is clear that it also satisfies \refeq{eq:MainThm} and then, arguing as in Section 8 of \cite{Tolsa-bilip}, one can prove that its image measure $\sigma = \varphi_{\#}\mu$ under a bi-Lipschitz map $\varphi$ satisfies
\begin{equation*}
\sigma(B)\leq C_\varphi r(B)^n
\end{equation*}
and
\begin{equation*}
\int_B \int_0^{r(B)}\beta_{\sigma,2}^n(x,r)^2\theta_{\sigma}^n(x,r)\frac{dr}{r}d\sigma(x)\lesssim C_\varphi \sigma(B),
\end{equation*}
for all balls $B$ of radius $r(B)$, where $C_\varphi$ is a positive constant only depending on the bi-Lipschitz constant of $\varphi$. Then, using Chebishev's inequality, one can prove that there exists an appropriate restriction $\tau$ of $\sigma$ with $||\tau|| \approx ||\sigma||$ and such that
\begin{equation*}
\sup_{x\in\R^{n+1}, R>0}\left\{\theta^n_\tau(x,R)+\int_0^\infty\beta_{\tau,2}(x,r)^2\theta_\tau^n(x,r)\frac{dr}{r}\right\}\leq C_\varphi.
\end{equation*}

\par\medskip
The plan of the paper is the following: in Section 2, we introduce some notation and recall some results that will be used throughout the text; the dyadic lattice of cells with small boundaries, constructed by David and Mattila, is introduced in Section 3, and the new Corona Decomposition by Azzam and Tolsa is introduced in Section 4; the proof of Theorem \ref{Th:MainThm}, in which we follow Tolsa's ideas \cite{Tolsa-pubMat} is carried out in Sections 5-9, and Corollary \ref{Cor:LipHarmCap} is proved in Section 10.
\section{Preliminaries}

\subsection{A useful estimate}
$ $
\par\medskip
Let $\mu$ be a positive Radon measure in $\R^d$ such that $\mu(B(x,r))\leq c_0r^n$ for all $x\in\R^d$ and all $r>0$. Then, for all $x\in\R^d$ and all $r>0$,
\begin{equation}
\label{Eq:CrecLineal}
\int_{|x-y|>r}\frac{d\mu(y)}{|x-y|^{n+1}} \leq \frac{c_0}{r}.
\end{equation}
This estimate, that can be easily proved by splitting the domain of integration into annuli $\{y\in\R^d\colon 2^k r <|y-x|\leq 2^{k+1}r\}$, $k\geq 0$ is commonly used in Calder\'{o}n-Zygmund theory, and we will also make use of it several times in this paper. 

\subsection{Notation}
\begin{itemize}
\item As it is usual in Harmonic Analysis, a letter $c$ will denote an absolute constant that may change its value at different occurrences. Constants with subscripts will retain their value at different occurrences. The notation $A\lesssim B$ means that there is a positive absolute constant $C$ such that $A\leq CB$, and $A\approx B$ is equivalent to $A\lesssim B\lesssim A$. 
\item If $B$ is a ball in $\R^d$, we denote its radius by $r(B)$. Given $\lambda>0$ the ball which is concentric with $B$ and has radius $\lambda r(B)$ is denoted by $\lambda B$.
\item If $\mu$ is a positive Radon measure in $\R^d$ and $B$ is a ball in $\R^d$, the average $n$-dimensional density of $B$ is 
\begin{equation*}
\theta^n_\mu(B) = \frac{\mu(B)}{r(B)^n},
\end{equation*}
so $\theta_{\mu}^n(B)=\theta_{\mu}^n(x,r)$ if $B=B(x,r)$. As $n$ will be fixed throughout the text, we will usually omit it to simplify the notation.
\item If $\mu$ is a Radon measure in $\R^d$ and $A\subset\R^d$, the restriction of $\mu$ to $A$ is denoted $\mu\lfloor_A$ or, simply, $\mu_A$, and it is defined by
\begin{equation*}
\mu\lfloor_A(E) = \mu(E\cap A).
\end{equation*}

\end{itemize}
\subsection{Suppressed operators}
$ $
\par\medskip
In this section, we recall the definition and most important properties of the so-called \emph{suppressed operators}, introduced by Nazarov, Treil and Volberg in \cite{NTV}, and that may be thought of as regular truncations of a Calder\'{o}n-Zygmund operator. All definitions and results in this section can be found in \cite{Volberg}.
\par\medskip
Let $k$ be an $n$-dimensional antisymmetric Calder\'{o}n-Zygmund kernel in $\R^d$. Given a non-negative and $1$-Lipschitz function $\Phi\colon\R^d\rightarrow\R$, we define
$$
k_\Phi(x,y)=k(x,y)\frac{1}{1+k(x,y)^2\Phi(x)^n\Phi(y)^n}.
$$

Then, $k_\Phi$ is also an antisymmetric Calder\'{o}n-Zygmund kernel, whose Calder\'{o}n-Zygmund constants do not depend on $\Phi$ but only on those of $k$, such that
\begin{enumerate}
\item $k_\Phi(x,y)=k(x,y)$ if $\Phi(x)\Phi(y)=0$.
\item $\displaystyle{|k_\Phi(x,y)|\leq c(n)\min\left\{\frac{1}{\Phi(x)^n},\frac{1}{\Phi(y)^n}\right\}}$.
\end{enumerate}

\par\medskip
We denote by $T_\Phi$ the integral operator associated to the kernel $k_\Phi$, that is, if $\nu$ is a signed Borel measure in $\R^d$ and $x\in \R^d$, 
$$
T_\Phi\nu(x) = \int k_\Phi(x,y)d\nu(y)
$$
whenever the integral makes sense. Naturally, we can also define the associated truncated operators
\begin{equation*}
T_{\Phi,\epsilon}\nu(x) = \int_{|x-y|>\epsilon}k_\Phi(x,y)d\nu(y)
\end{equation*}
and the maximal operator
\begin{equation*}
T_{\Phi,*}\nu(x) = \sup_{\epsilon>0}|T_{\Phi,\epsilon}\nu(x)|.
\end{equation*}
\par\medskip

We also introduce the maximal operator associated to $\Phi$
\begin{equation*}
M^r_\Phi \nu(x) = \sup_{r\geq \Phi (x)}\frac{|\nu|[B(x,r)]}{r^n}.
\end{equation*}

As usual, if $\sigma$ is any fixed positive Borel measure in $\R^d$, we can make these operators act on measures of the form $f\sigma$. To simplify notation, we denote, in such a case,
\begin{equation*}
T_{\sigma,\Phi}f = T_\Phi(f\sigma),\;\; T_{\sigma,\Phi,\epsilon}f = T_{\Phi,\epsilon}(f\sigma), \;\; M^r_{\sigma,\Phi}f = M^r_{\Phi}(f\sigma).
\end{equation*}
\begin{otherl} 
\label{Le:TruncadosYMaximal}
Let $\nu$ be a signed and finite Borel measure in $\R^d$ and $x\in \R^d$.
\begin{enumerate}
\item If $\epsilon > \Phi(x)$,
\begin{equation*}
|T_{\Phi,\epsilon}\nu(x)-T_{\epsilon}\nu(x)|\lesssim M^r_\Phi\nu(x).
\end{equation*}
\item If $\epsilon\leq\Phi(x)$,
\begin{equation*}
|T_{\Phi,\epsilon}\nu(x)-T_{\Phi,\Phi(x)}\nu(x)|\lesssim M^r_\Phi\nu(x).
\end{equation*}
\end{enumerate}
\end{otherl}
Finally, we state a Cotlar-type inequality that will be especially useful when dealing with suppressed operators $T_\Phi$. To do so, we introduce a couple more of maximal operators associated to any positive Radon measure $\sigma$ in $\R^d$: for $f\in L^1_{loc}(\sigma)$ and $x\in\R^d$,
\begin{equation*}
\tilde{M}_\sigma f(x) = \sup_{r>0}\frac{1}{\sigma[B(x,3r)]}\int_{B(x,r)}|f|d\sigma,\;\; \tilde{M}_{\sigma,\frac{3}{2}}f(x) = \sup_{r>0}\left(\frac{1}{\sigma[B(x,3r)]}\int_{B(x,r)}|f|^{\frac{3}{2}d\sigma}\right)^{\frac{2}{3}}.
\end{equation*}

\begin{other}
\label{Th:CotlarNTV}
Let $\sigma$ be a positive Radon measure in $\R^d$, and let, for $x\in\R^d$,
\begin{equation*}
\mathcal{R}(x) = \sup\{r>0\colon \sigma[B(x,r)]>C_0r^n\},
\end{equation*}
where $C_0>0$ is some fixed constant. Let $S$ be a singular integral operator with Calder\'{o}n-Zygmund kernel $s$, with
$$
|s(x,y)|\lesssim \min\left\{\frac{1}{\mathcal{R}(x)^n},\frac{1}{\mathcal{R}(y)^n}\right\}.
$$
and such that $S_\sigma$ is bounded in $L^2(\sigma)$. Then, for all $f\in L^1_{loc}(\sigma)$ and all $x\in\R^d$,
\begin{equation*}
S_*(f\sigma)(x) \lesssim \tilde{M}_\sigma(S(f\sigma))(x) + \tilde{M}_{\sigma,\frac{3}{2}}f(x).
\end{equation*}

\end{other}

\section{The dyadic lattice of cells with small boundaries}
We will use the dyadic lattice of cells with small boundaries constructed by David and Mattila in \cite[Theorem 3.2]{David-Mattila}.
The properties of this dyadic lattice are summarized in the next lemma.

\par\medskip

\begin{otherl}[David, Mattila]
\label{Le:DyadicGrid}
Let $\mu$ be a Radon measure on $\R^d$, $E=\supp(\mu)$, and consider two constants $C_0>1$ and $A_0>5000\,C_0$. Then, there exists a sequence $\{\D_k\}_{k=0}^{\infty}$ of families of Borel subsets of $E$ with the following properties:
\begin{itemize}
\item For each integer $k\geq0$, $\D_k$ is a partition of $E$, that is, the sets $Q\in\D_k$ are pairwise disjoint and 
$$
\bigcup_{Q\in\D_k}Q = E.
$$
\par\medskip

\item If $k,l$ are integers, $0\leq k < l$, $Q\in\D_k$ and $R\in\D_l$, then either $R\subset Q$ or $Q\cap R = \emptyset$.

\item The general position of the cells $Q$ can be described as follows: for each $k\geq 0$ and each cell $Q\in\D_k$, there is a ball $B(Q)=B(z_Q,r(Q))$ such that
$$
z_Q\in E, \qquad A_0^{-k}\leq r(Q)\leq C_0\,A_0^{-k}, \qquad E\cap B(Q)\subset Q\subset E\cap 28\,B(Q),
$$
where the balls $5B(Q)$, $Q\in\D_k$, are pairwise disjoint.

\par\medskip
\item The cells $Q\in\D_k$ have small boundaries, that is, for each $Q\in\D_k$ and each
integer $l\geq0$, set
$$N_l^{ext}(Q)= \{x\in E\setminus Q:\,\dist(x,Q)< A_0^{-k-l}\},$$
$$N_l^{int}(Q)= \{x\in Q:\,\dist(x,E\setminus Q)< A_0^{-k-l}\},$$
and
$$N_l(Q)= N_l^{ext}(Q) \cup N_l^{int}(Q).$$
Then
\begin{equation}\label{eqsmb2}
\mu(N_l(Q))\leq (C^{-1}C_0^{-3d-1}A_0)^{-l}\,\mu(90B(Q)).
\end{equation}
\par\medskip

\item Denote by $\D_k^{db}$ the family of cells $Q\in\D_k$ for which
\begin{equation}\label{eqdob22}
\mu(100B(Q))\leq C_0\,\mu(B(Q)).
\end{equation}
Then, for all $Q\in\D_k\setminus\D_k^{db}$ , we have that $r(Q)=A_0^{-k}$ and $\mu[100B(Q)]\leq C_0^{-l}\mu[100^{l+1}B(Q)]$ for all $l\geq 1$ such that $100^l\leq C_0$.
\end{itemize}
\end{otherl}

\par\medskip
We use the notation $\D=\bigcup_{k\geq0}\D_k$. For $Q\in\D$, we set $\D(Q) =\{P\in\D:P\subset Q\}$.
Given $Q\in\D_k$, we denote $J(Q)=k$. We set $\l(Q)= 56\,C_0\,A_0^{-k}=\l_k$ and we call it the side length of $Q$. Note that 
$$
\frac1{28}\,C_0^{-1}\l(Q)\leq \diam(Q)\leq\l(Q).
$$
Observe that $r(Q)\approx\diam(Q)\approx\l(Q)$. In addition, we call $z_Q$ the center of $Q$, and we call the cell $Q'\in \D_{k-1}$ such that $Q'\supset Q$ the parent of $Q$. We set $B_Q=28\,B(Q)$, so that 
$$
E\cap \tfrac1{28}B_Q\subset Q\subset B_Q.
$$
\par\medskip
We assume $A_0$ to be big enough so that the constant $C^{-1}C_0^{-3d-1}A_0$ in 
\refeq{eqsmb2} satisfies 
$$C^{-1}C_0^{-3d-1}A_0>A_0^{1/2}>10.$$
Then we deduce that, for all $0<\lambda\leq1$,
\begin{equation}\label{eqfk490}
\mu\bigl(\{x\in Q:\dist(x,E\setminus Q)\leq \lambda\,\ell(Q)\}\bigr) + 
\mu\bigl(\bigl\{x\in 4B_Q\setminus Q:\dist(x,Q)\leq \lambda\,\ell(Q)\}\bigr)\leq
c\,\lambda^{1/2}\,\mu(3.5B_Q).
\end{equation}

We denote
$\D^{db}=\bigcup_{k\geq0}\D_k^{db}$ and $\D^{db}(Q) = \D^{db}\cap \D(Q)$.
Note that, in particular, from \refeq{eqdob22} we obtain
$$\mu(100B(Q))\leq C_0\,\mu(Q)\qquad\mbox{if $Q\in\D^{db}.$}$$
For this reason we will call the cells from $\D^{db}$ doubling. 

As shown in \cite[Lemma 5.28]{David-Mattila}, any cell $R\in\D$ can be covered $\mu$-a.e.\
by a family of doubling cells:
\par\medskip

\begin{otherl}
\label{Le:CubosCubiertosPorDoblantes}
Let $R\in\D$. Suppose that the constants $A_0$ and $C_0$ in Lemma \ref{Le:DyadicGrid} are
chosen suitably. Then there exists a family of
doubling cells $\{Q_i\}_{i\in I}\subset \D^{db}$, with
$Q_i\subset R$ for all $i$, such that their union covers $\mu$-almost all $R$.
\end{otherl}

\par\medskip
From now on we will assume that $C_0$ and $A_0$ are some big fixed constants so that the
results stated in the lemmas of this section hold.

\section{The corona decomposition}
Let $\mu$ be any measure satisfying the same hypotheses as the one in Theorem \ref{Th:MainThm} (e.g., the restriction of the measure $\mu$ presented there to any ball $B$) and construct the dyadic lattice $\D$ of cells with small boundaries associated to $\mu$ that is given by Lemma \ref{Le:DyadicGrid}. Let $R_0\in\D$ be such that $\supp(\mu)\subset R_0$ and $\diam(\supp(\mu))\leq \l(R_0)$ (we can assume, without loss of generality, that $\D_0 = \{R_0\}$), and let $\Top$ be a family of doubling cells contained in $R_0$ and such that $R_0\in\Top$ that we will fix below.
\par\medskip
For every $R\in\Top$, denote by $\Stop(R)$ the family of maximal cells $Q\in\Top$ that are contained in $R$, and by $\Tree(R)$ the family of cells $Q\in\mathcal{D}$ that are contained in $R$ and not contained in any $Q'\in\Stop(Q)$. Then, we define
\begin{equation*}
\Good(R) = R\setminus \bigcup_{Q\in\Stop(R)}Q
\end{equation*}
and, for $Q\subset R$,
\begin{equation*}
\dmu(Q,R)=\int_{2B_R\setminus Q}\frac{d\mu(y)}{|y-z_Q|^{n}}.
\end{equation*}
\par\medskip
The arguments of Azzam and Tolsa \cite[Lemma 7.2]{Azzam-Tolsa} can be easily adapted to prove the following:
\begin{otherl}
\label{Th:CoronaDesc}
There exists a family $\Top\subset \D^{db}$ as above such that, for all $R\in\Top$, there exists a bi-Lipschitz injection $g_R\colon\R^n\rightarrow\R^d$ with the bi-Lipschitz constant bounded above by some absolute constant and with image $\Gamma_R=g(\R^n)$ such that
\begin{enumerate}
\item $\mu$-almost all $\Good(R)$ is contained in $\Gamma_R$.
\item For all $Q\in\Stop(R)$ there exists another cell $\tilde{Q}\in\D(R)$ with $Q\subset \tilde{Q}$ such that $\dmu(Q,\tilde{Q})\leq c\,\theta_\mu(B_R)$ and $B_{\tilde{Q}}\cap \Gamma_R \neq \emptyset$.
\item For all $Q\in\Tree(R)$, $\theta_{\mu}(1.1B_Q)\leq c\,\theta_{\mu}(B_R)$.
\end{enumerate}
Furthermore, the cells $R\in\Top$ satisfy the following packing condition:
\begin{equation*}
\sum_{R\in\Top}\theta_{\mu}(B_R)^2\mu(R)\lesssim \theta_{\mu}(B_{R_0})^2\mu(R_0)+\iint_{0}^{\l(R_0)}\beta^n_{\mu, 2}(x,r)^2\theta^n_{\mu}(x,r)\frac{dr}{r}d\mu(x).
\end{equation*}
\end{otherl}

\section{The main lemma}
For technical reasons, we will assume that the kernel $k$ of $T$ is not only in $\mathcal{K}^n(\R^d)$, but that it is also a bounded function, so that the definition of $T\mu(x)$ makes perfect sense for all $x\in\R^d$ if $\mu$ is a finite and compactly supported Borel measure in $\R^d$, which is the case we are considering. However, as all of our estimates will be independent of the $L^{\infty}$ norm of $k$, our result can be easily extended for general Calder\'{o}n-Zygmund kernels $k\in\mathcal{K}^n(\R^d)$ by a standard smoothing procedure (see, for example, equation $(44)$ in \cite{Tolsa-LittlewoodPaley}). 

\par\medskip

The following sections will be devoted to proving this result:

\begin{lemma}[Main Lemma]
\label{Le:MainLemma}
Let $\mu$ be a positive Radon measure in $\R^d$  with compact support and polynomial growth of degree $n$. Then,
\begin{equation*}
||T\mu||^2_{L^2(\mu)}\lesssim ||\mu|| + \iint\beta^n_{\mu,2}(x,r)^2\theta^n_\mu(x,r)\frac{dr}{r}d\mu(x).
\end{equation*}
\end{lemma}
\par\medskip
Theorem \ref{Th:MainThm} follows from the non-homogeneous $T(1)$ theorem \cite[Theorem 1.1 and Lemma 7.3]{Tolsa-LittlewoodPaley} and the previous lemma, as it enables us to estimate $||T(\chi_B\mu)||_{L^2(\chi_B\mu)}$ for all balls $B\subset\R^{d}$. Indeed, if $\mu$ is the measure from Theorem \ref{Th:MainThm}, $B$ is a ball in $\R^d$ and $r(B)$ is its radius, applying Lemma \ref{Le:MainLemma} to the measure $\chi_B\mu$, we obtain 
\begin{equation*}
||T(\chi_B\mu)||^2_{L^2(\mu)}\lesssim \mu(B) + \iint\beta^n_{\chi_B\mu,2}(x,r)^2\theta^n_{\chi_B\mu}(x,r)\frac{dr}{r}d\mu(x)\lesssim \mu(B),
\end{equation*}
where the last inequality follows directly from the hypotheses of Theorem \ref{Th:MainThm}. Therefore, the non-homogeneous $T(1)$ theorem applies, and we obtain that $T_\mu$ is bounded in $L^2(\mu)$.
\par\medskip
To prove the Main Lemma, we will closely follow the ideas by Tolsa in \cite{Tolsa-pubMat}, but we will use the dyadic lattice $\D$ associated to $\mu$, which is introduced in Section 3, instead of the usual dyadic lattice of \emph{true} cubes in $\R^d$. We apply Lemma \ref{Th:CoronaDesc} to obtain a Corona Decomposition for $\mu$, and we decompose $T\mu$ in terms of that Corona Decomposition, since the terms that arise from that decomposition will be tractable. The main difference between our proof and Tolsa's one will be found in Section $8$, since the fact that the cells in $\D$ have thin boundaries helps us to avoid going through the process of averaging over random dyadic lattices to get the estimate that is proved there.

\section{Decomposition of $T\mu$ with respect to the corona decomposition}
To estimate $\ntmu^2$ we will decompose $T\mu$ with respect to the corona decomposition from Theorem \ref{Th:CoronaDesc}. To do so, let $\psi$ be a non-negative and radial $\mathcal{C}^{\infty}$ function such that 
\begin{equation*}
\chi_{B(0,0.001)}\leq \psi \leq \chi_{B(0,0.01)} \;\;\;\;\; \text{ and }\;\;\;\;\; ||\nabla\psi||\lesssim 1.
\end{equation*}
\par\medskip
For each $k\in\Z$, define $\psi_k(z) = \psi(A_0^k z)$ and $\varphi_k = \psi_k - \psi_{k+1}$, so that each function $\varphi_k$ is non-negative and supported on $B(0,0.01A_0^{-k})\setminus B(0,0.001A_0^{-k-1})$ and, furthermore,
\begin{equation*}
\sum_{k\in\Z}\varphi_k (z) = 1
\end{equation*}
for all $x\in \R^{d}\setminus\{0\}$.
\par\medskip
Now observe that, for $x\in\supp(\mu)$ we have
\begin{equation*}
\begin{aligned}
T\mu(x) &= \int k(x,y)d\mu(y) = \int \left(\sum_{k\in\Z}\varphi_k(x-y)\right)k(x,y)d\mu(y) \\
		&= \sum_{k\in\Z}\int\varphi_{k}(x-y)k(x,y)d\mu(y).
\end{aligned}
\end{equation*}
Therefore, if we define
\begin{equation*}
T_k\mu(x) = \int \varphi_k(x-y)k(x,y)d\mu(y)
\end{equation*}
we have
\begin{equation*}
T\mu(x) = \sum_{k\in\Z}T_k\mu(x).
\end{equation*}
\par\medskip
Now set $\D_k=\{R_0\}$ whenever $k<0$ and $T_Q\mu = \chi_Q T_{J(Q)}\mu$ for all $Q\in\D$. Then,
\begin{equation*}
\begin{aligned}
T\mu &= \sum_{k\in\Z}T_k\mu = \sum_{k\in\Z}\left(\sum_{Q\in\D_k}\chi_Q T_k\mu\right) \\
	 &= \sum_{k\in\Z}\sum_{Q\in\D_k} \chi_Q T_{J(Q)}\mu = \sum_{Q\in\D}T_Q\mu \\
	 &= \sum_{Q\in\mathcal{F}}T_Q\mu + \sum_{R\in\Top}\left(\sum_{Q\in\Tree(R)}T_Q\mu\right)\\
	 & = \sum_{Q\in\mathcal{F}}T_Q\mu + \sum_{R\in\Top} K_R\mu,
\end{aligned}
\end{equation*}
where, for $R\in\Top$,
\begin{equation*}
K_R\mu = \sum_{Q\in\Tree(R)}T_Q\mu
\end{equation*}
and $\mathcal{F}$ is a finite family of cells $Q\in\D$ with $\l(Q)\approx \diam(\supp(\mu))$.
\par\medskip
Notice that for $Q\in\mathcal{F}$, the estimate
\begin{equation*}
||T_Q\mu||_{L^2(\mu)}^2\lesssim ||\mu||
\end{equation*}
holds trivially. Therefore,
\begin{equation*}
\ntmu^2 \lesssim ||\mu|| + \left|\left| \sum_{R\in\Top}K_R\mu\right|\right|_{L^2(\mu)}^2 = \sum_{R\in\Top}||K_R\mu||_{L^2(\mu)}^2 + \sum_{R,R'\in \Top\colon R\neq R'}\langle K_R\mu, K_{R'}\mu \rangle_{\mu},
\end{equation*}
where $\langle \cdot, \cdot \rangle_\mu$ denotes the usual pairing in $L^2(\mu)$, i.e.,
$$
\langle f, g \rangle_\mu = \int fg d\mu
$$
\par\medskip

The \emph{diagonal sum} $\sum_{R\in\Top}||K_R\mu||^2_{L^2(\mu)}$ will be estimated in Section 7 using the fact that, on each $\Tree(R)$, $\mu$ can be approximated by a measure of the form $\eta\H^n_{\Gamma_R}$, where $\eta$ is a bounded function, and $T_{\H^n_{\Gamma_R}}$ is bounded in $L^2(\H^n_{\Gamma_R})$ because $\Gamma_R$ is a bi-Lipschitz image of $\R^n$, and thus uniformly $n$-rectifiable (see \cite{Tolsa-Alphas}, or the more classical reference \cite{David-Semmes-OnAndOf} for the case where $K$ is assumed to be $\mathcal{C}^{\infty}$ away from the origin). To deal with the \emph{non-diagonal sum} $\sum_{R,R'\in\Top\colon R\neq R'}\langle K_R\mu, K_{R'}\mu\rangle_\mu$, we will use quasi-orthogonality arguments. Here, the fact that the cells from $\D$ have thin boundaries will be crucial.

\section{The estimate of $\sum_{R\in\Top}||K_R\mu||_{L^2(\mu)}^2$}

The goal of this section is to prove the following:

\begin{lemma}
\label{Le:Diagonal}
\begin{equation*}
\sum_{R\in\Top}||K_R\mu||^2_{L^2(\mu)} \lesssim \sum_{R\in\Top}\theta_\mu(B_R)^2\mu(R).
\end{equation*}
\end{lemma}

\subsection{Regularization of the stopping squares}
$ $
\par\medskip
Pick $R\in \Top$ and define
\begin{equation*}
d_R(x) = \inf_{Q\in\Tree(R)}\left\{|x-z_Q|+\l(Q)\right\}.
\end{equation*}
Notice that $d_R$ is a $1$-Lipschitz function because it is defined as the infimum of a family of $1$-Lipschitz functions. 
\par\medskip
Now, we denote
\begin{equation}
\label{eq:DefBo(R)}
B_0(R) = B(z_R, 29A_0^{-J(R)}), \;\; W_R = \{x\in \R^d \colon d_R(x) = 0\}
\end{equation}
and, for all $x\in B_0(R)\setminus W_R$, we denote by $Q_x$ the largest cell $Q_x\in\D$ containing $x$ and such that
\begin{equation*}
\l(Q_x)\leq \frac{1}{60} \inf_{y\in Q_x} d_R(y).
\end{equation*}
We define $\Reg(R)$ as the family of the cells $\{Q_x\}_{x\in B_0(R)\setminus W_R}$, which are pairwise disjoint. Note that
\begin{equation*}
B_0(R)\setminus \bigcup_{Q\in\Reg(R)}Q = W_R \subset \Good(R).
\end{equation*}

\par\medskip

\begin{lemma}
Properties of the regularized stopping cells:
\label{Le:PropCubosRegulares}
\begin{enumerate}
\item If $Q\in \Reg(R)$ and $x\in B(z_Q, 50\l(Q))$, then $d_R(x) \approx \l(Q)$.
\item If $Q,Q'\in \Reg(R)$ are such that $B(z_Q, 50\l(Q))\cap B(z_{Q'}, 50\l(Q'))\neq \emptyset$, then $\l(Q)\approx \l(Q')$.
\item If $Q\in\Reg(R)\cap\D(R)$, there exists $Q'\in \Stop(R)$ such that $Q\subset Q'$.
\item If $Q\in\Reg(R)$, $x\in Q$ and $r>\l(Q)$, then
\begin{equation*}
\mu[B(x,r)\cap B_R] \lesssim \theta_\mu(B_R)r^n.
\end{equation*}
\end{enumerate}

\end{lemma}

\begin{proof}
\begin{enumerate}

\item First, observe that by definition of $\Reg(R)$,
\begin{equation*}
Q\in \Reg(R) \Rightarrow \l(Q)\leq \frac{1}{60}\inf_{y\in Q}d_R(y) \leq \frac{1}{60}d_R(z_Q),
\end{equation*}
that is, $d_R(z_Q)\geq 60\l(Q)$. Therefore, since $d_R$ is $1$-Lipschitz and $|x-z_Q|\leq 50\l(Q)$,
\begin{equation*}
d_R(x)\geq d_R(z_Q)-|x-z_Q|\geq 60\l(Q)-50\l(Q)=10\l(Q).
\end{equation*}
\par\medskip
On the other hand, again by definition of $\Reg(R)$, we have
\begin{equation*}
\l(\hat{Q})>\frac{1}{60}\inf_{y\in\hat{Q}}d_R(y),
\end{equation*}
where $\hat{Q}$ is the parent of $Q$. Then, there exists $\hat{y}\in\hat{Q}$ such that
\begin{equation*}
d_R(\hat{y})<60\l(\hat{Q})=60A_0\l(Q).
\end{equation*}
Now, since $x,\hat{y}\in\hat{Q}$ and $\diam(\hat{Q})\leq \l(\hat{Q}) = A_0\l(Q)$, and taking into account once again that $d_R$ is $1$-Lipschitz, we get
\begin{equation}
\label{Eq:dRx}
d_R(x)\leq d_R(\hat{y})+|x-\hat{y}|\leq 60A_0\l(Q)+A_0\l(Q) = 61A_0\l(Q),
\end{equation}
as desired.
\par\medskip
\item This follows directly from $(1)$.
\par\medskip
\item If such a $Q'\in\Stop(R)$ does not exist, we get that $Q\in\Tree(R)$. Then, for all $x\in Q$,
\begin{equation*}
d_R(x) \leq \inf_{Q'\in\Tree(R)}\left[|x-z_{Q'}|+\l(Q')\right] \leq |x-z_Q|+\l(Q) \leq 2\l(Q).
\end{equation*}
However, since $Q\in\Reg(R)$, we get
\begin{equation*}
\l(Q)\leq \frac{1}{60}\inf_{x\in Q}d_R(x),
\end{equation*}
so $d_R(x)\geq 60\l(Q)$ for all $x\in Q$. This is a contradiction.
\par\medskip
\item Since $x\in Q$ and $Q\in\Reg(R)$, by \refeq{Eq:dRx} we have $d_R(x)<62A_0\l(Q)$. Now, since
\begin{equation*}
d_R(x) = \inf_{Q'\in\Tree(R)}\left[|x-z_{Q'}|+\l(Q')\right]
\end{equation*}
we obtain that there exists $Q'\in\Tree(R)$ such that
\begin{equation*}
|x-z_{Q'}|+\l(Q')<62A_0\l(Q).
\end{equation*}
From this, we get
\begin{equation*}
|x-z_{Q'}|<62A_0r \;\;\text{and}\;\; r>\frac{1}{62A_0\l(Q')}
\end{equation*}
and, therefore, we have two possibilities:
\par\medskip
\begin{enumerate}
\item There exists $Q''\in\Tree(R)$ with $Q'\subset Q''$ and $\l(Q'')\lesssim r$ such that $B(x,r)\subset 1.1B_{Q''}$. In such a case, since $Q''\in\Tree(R)$, we have $\theta_\mu(1.1B_{Q''})\lesssim \theta_\mu(B_R)$, and therefore
\begin{equation*}
\begin{aligned}
\mu[B(x,r)\cap B_R] &\leq \mu[B(x,r)] \leq \mu(1.1B_{Q''}) = \theta_\mu(1.1 B_{Q''})r(B_{Q''})^n \\
			        & \lesssim \theta_\mu(1.1B_{Q''})r^n \lesssim \theta_\mu(B_R)r^n.
\end{aligned}
\end{equation*}
\par\medskip
\item $B(x,r)\supset B_R$. In this case,
\begin{equation*}
\mu[B(x,r)\cap B_R] = \mu(B_R) = \theta_\mu(B_R)r(B_R)^n \leq \theta_\mu(B_R)r^n.
\end{equation*}
\end{enumerate}

\end{enumerate}
\end{proof}

\subsection{The suppressed operators $T_{\Phi_R}$.}
$ $
\par\medskip
Fix $R\in\Top$ and define 
\begin{equation*}
\Phi_R(x) = \frac{1}{20A_0^2}d_R(x).
\end{equation*}

\begin{lemma}
Properties of the suppressing function $\Phi_R$:
\label{Le:PropiedadesPhiR}
\begin{enumerate}
\item If $x \in Q$ for some $Q\in \Stop(R)$, $\Phi_R(x)\leq \frac{1}{10A_0}\l(Q)$.
\item If $x \in \Good(R)$, $\Phi_R(x) = 0$.
\item If $x\in Q$ for some $Q\in\Reg(R)$, then $\Phi_R(x)\gtrsim \l(Q)$.
\item For all $x\in B_R$ and all $r\geq \Phi_R(x)$, 
\begin{equation}
\label{Eq:DensBolasRadioGrande}
\mu[B(x,r)\cap B_R] \leq C_1\,\theta_\mu(B_R)r^n.
\end{equation}
\end{enumerate}
\end{lemma}

\begin{proof}
\begin{enumerate}
\item Let $Q\in\Stop(R)$ and $x\in Q$. We have
\begin{equation*}
d_R(x) = \inf_{Q'\in\Tree(R)}\left[|x-z_{Q'}|+\l(Q')\right] \leq |x-z_{\hat{Q}}|+\l(\hat{Q}),
\end{equation*}
where $\hat{Q}$ is the parent of $Q$. Then,
\begin{equation*}
\Phi_R(x) = \frac{1}{20A_0^2}d_R(x) \leq \frac{1}{20A_0^2}2\l(\hat{Q}) = \frac{1}{10A_0^2}A_0\l(Q) = \frac{1}{10A_0}\l(Q).
\end{equation*}
\par\medskip
\item If $x\in\Good(R)$, there exist arbitrarily small cells $Q\in\Tree(R)$ that contain $x$. Therefore,
\begin{equation*}
\Phi_R(x) = \frac{1}{20A_0^2}\inf_{Q\in\Tree(R)}\left[|x-z_Q|+\l(Q)\right] = 0.
\end{equation*}
\par\medskip
\item This follows directly from $(1)$ in Lemma \ref{Le:PropCubosRegulares}.
\par\medskip
\item First, observe that if $x\in R\setminus \bigcup_{Q\in\Reg(R)}Q$, then \refeq{Eq:DensBolasRadioGrande} holds for all $r>0$, and this can be proved arguing as in $(4)$ in Lemma \ref{Le:PropCubosRegulares} and taking into account that $d_R(x)=0$. Otherwise, if $x\in Q$ for some $Q\in\Reg(R)$, by $(1)$ in lemma \ref{Le:PropCubosRegulares} we have that $r\gtrsim \l(Q)$, and so $(4)$ in lemma \ref{Le:PropCubosRegulares} applies.

\end{enumerate}
\end{proof}

\begin{lemma}
\label{Le:DesigualdadPuntual}
For $x\in R$,
$$
|K_R\mu(x)|\leq T_{\Phi_R, *}(\chi_{B_0(R)}\mu)(x) + c\theta_\mu(B_R),
$$
where $B_0(R)=B(z_R, 29A_0^{-J(R)})$, which is defined in \refeq{eq:DefBo(R)}, satisfies $\theta_\mu(B_0(R))\approx \theta_\mu(B_R)$.
\end{lemma}

\begin{proof}
The fact that $\theta_\mu(B_0(R))\approx \theta_\mu(B_R)$ follows immediately from $R\in\D^{db}$.
\par\medskip
Recall that
\begin{equation*}
K_R\mu = \sum_{Q\in\Tree(R)}T_Q\mu = \sum_{Q\in\Tree(R)}\chi_Q T_{J(Q)}\mu.
\end{equation*}
Now, for $x\in R$, we have two possibilities: either $x\in Q$ for some $Q\in\Stop(R)$ or $x\in\Good(R)$.

\begin{enumerate}
\item Suppose $x\in Q$ for some $Q\in\Stop(R)$. Then, 

\begin{equation*}
\begin{aligned}
|K_R\mu(x)| &= \left|\sum_{j=J(R)}^{J(Q)-1}T_j\mu(x)\right| = \left|\int\left(\sum_{j=J(R)}^{J(Q)-1}\varphi_j(x-y)\right)k(x,y)d\mu(y)\right| \\
	&= \left|\int[\psi_{J(R)}(x-y)-\psi_{J(Q)}(x-y)]k(x,y)d\mu(y)\right|\\
	&= \left|\int_{|y-x|\geq 0.001A_0^{-J(Q)-1}}[\psi_{J(R)}(x-y)-\psi_{J(Q)}(x-y)]k(x,y)\chi_{B_0(R)}(y)d\mu(y)\right|\\
	&\leq |T_{2A_0^{-1}\l(Q)}(\chi_{B_0(R)}\mu)(x)|+c\theta_\mu(B_R)\\
	&\leq |T_{\Phi_R, 2A_0^{-1}\l(Q)}(\chi_{B_0(R)}\mu)(x)| + |T_{2A_0^{-1}\l(Q)}(\chi_{B_0(R)}\mu)(x) -T_{\Phi_R, 2A_0^{-1}\l(Q)}(\chi_{B_0(R)}\mu)(x)| +c\theta_\mu(B_R)\\
	&\leq T_{\Phi_R,*}(\chi_{B_0(R)}\mu)(x) + M^r_{\Phi_R}(\chi_{B_0(R)}\mu)(x)+c\theta_\mu(B_R)\\
	&\leq T_{\Phi_R,*}(\chi_{B_0(R)}\mu)(x)+c\theta_\mu(B_R),
\end{aligned}
\end{equation*}
where the penultimate inequality follows from the fact that $\Phi_R(x)\leq 2A_0^{-1}\l(Q)$ and the last one from lemma \ref{Le:TruncadosYMaximal}.
\par\medskip
\item If $x\in\Good(R)$, we have
$$
|K_R\mu(x)| = \lim_{N\rightarrow\infty} \left|\int[\psi_{J(R)}(x-y)-\psi_N(x-y)]k(x,y)d\mu(y)\right|.
$$
Then, for $N>J(R)$ we obtain, arguing as above, that
\begin{equation*}
\begin{aligned}
\left|\int[\psi_{J(R)}(x-y)-\psi_N(x-y)]k(x,y)d\mu(y)\right| &\leq |T_{2\l_{N+1}}(\chi_{B_0(R)}\mu)(x)|+c\theta_\mu(B_R) \\
	&\leq |T_{2\l_{N+1}}(\chi_{B_0(R)}\mu)(x)-T_{\Phi_R, 2\l_{N+1}}(\chi_{B_0(R)}\mu)(x)| \\
		& \qquad+ |T_{\Phi_R,2\l_{N+1}}(\chi_{B_0(R)}\mu)(x)|+c\theta_\mu(B_R)\\
	&\leq M^r_{\Phi_R}(\chi_{B_0(R)}\mu)(x) + T_{\Phi_R,*}(\chi_{B_0(R)}\mu)(x) + c\theta_\mu(B_R)\\
	&\leq T_{\Phi_R,*}(\chi_{B_0(R)}\mu)(x) + c\theta_\mu(B_R)
\end{aligned}
\end{equation*}
where in the penultimate inequality we used the fact that $\Phi_R(x) = 0 \leq 2\l_{N+1}$. Then, letting $N\rightarrow\infty$, we obtain
$$
|K_R\mu(x)|\leq T_{\Phi_R,*}(\chi_{B_0(R)}\mu)(x) + c\theta_\mu(B_R),
$$
as desired.
\end{enumerate}
\end{proof}

\subsection{A Cotlar-type inequality.}
$ $
\par\medskip
\begin{lemma}
\label{Le:Cotlar}
Let $R\in\Top$. Then, for all $0<s\leq 1$,
\begin{equation}
\label{eq:cotlar}
T_{\Phi_R, *}(f\H^n\lfloor_{\Gamma_R})(x) \leq C_s\left[M^r_{\Phi_R}((T_*(f\H^n\lfloor_{\Gamma_R})^s)\H^n\lfloor_{\Gamma_R})(x)^{\frac{1}{s}}+M^r_{\Phi_R}(f\\H^n\lfloor_{\Gamma_R})(x)\right]
\end{equation}
for all $x\in B_0(R)$.
\end{lemma}

\begin{proof}
Denote $\nu = f\H^n\lfloor_{\Gamma_R}$. We will prove that for all $x\in B_0(R)$ and all $\epsilon>0$,
\begin{equation*}
T_{\Phi_R,\epsilon}\nu(x) \leq C_s \left[M^r_{\Phi_R}((T_*\nu)^s \H^n\lfloor_{\Gamma_R})(x)^{\frac{1}{s}}+M^r_{\Phi_R}\nu(x)\right]
\end{equation*}
By $(2)$ in Lemma \ref{Le:TruncadosYMaximal}, we can limit ourselves to the case $\epsilon\geq \Phi_R(x)$. Furthermore, we can assume $\epsilon > \epsilon_0 := 0.9\dist(x,\Gamma_R)$ since otherwise $T_{\Phi_R, \epsilon}\nu(x) = T_{\Phi_R, \epsilon_0}\nu(x)$. Therefore, from now on we will assume $\epsilon\geq \max\{\Phi_R(x),0.9\dist(x,\Gamma_R)\}$. Notice that, in such a case, $\H^n(B(x,2\epsilon)\cap\Gamma_R)\gtrsim \epsilon^n$. We claim now that, for all $x'\in B(x,2\epsilon)\cap\Gamma_R)$

\begin{equation}
\label{eq:APorCotlar}
|T_{\Phi_R,\epsilon}\nu(x)|\leq |T_{\epsilon}\nu(x')|+CM^r_{\Phi_R}\nu(x).
\end{equation}

From this, the desired result follows easily. Indeed, this implies that for all $0<s\leq 1$, 
\begin{equation*}
|T_{\Phi_R,\epsilon}\nu(x)|^s\leq T_*\nu(x')^s+CM^r_{\Phi_R}\nu(x)^s,
\end{equation*}
and so, taking the $\H^n\lfloor_{\Gamma_R}$-average for with respect to $x'\in B(x,2\epsilon)$, we get

\begin{equation*}
\begin{aligned}
|T_{\Phi_R,\epsilon}\nu(x)|^s &\leq \frac{1}{\H^n[B(x,2\epsilon)\cap\Gamma_R]}\int_{B(x,2\epsilon)}T_*\nu(x')^sd\H^n\lfloor_{\Gamma_R}(x')+CM^r_{\Phi_R}\nu(x)^s \\
	& \lesssim \frac{1}{\epsilon^n}\int_{B(x,2\epsilon)}T_*\nu(x')^sd\H^n\lfloor_{\Gamma_R}(x')+M^r_{\Phi_R}\nu(x)^s\\
	& \lesssim M^r_{\Phi_R}((T_*\nu)^s\H^n\lfloor_{\Gamma_R})(x)+M^r_{\Phi_R}\nu(x)^s
\end{aligned}
\end{equation*}
and, exponentiating by $\frac{1}{s}$, \refeq{eq:cotlar} follows.
\par\medskip
Let us prove now \refeq{eq:APorCotlar}. We have
\begin{equation*}
|T_{\Phi_R,\epsilon}\nu(x)| \leq |T_{\Phi_R,\epsilon}\nu(x)-T_\epsilon\nu(x)|+|T_\epsilon\nu(x)|\lesssim |T_\epsilon\nu(x)| + M^r_{\Phi_R}\nu(x)
\end{equation*}
by Lemma \ref{Le:TruncadosYMaximal}, since $\epsilon>\Phi_R(x)$. Now, for all $x'\in B(x,2\epsilon)$
\begin{equation*}
\begin{aligned}
|T_\epsilon\nu(x)| &\leq |T_\epsilon\nu(x)-T_{4\epsilon}\nu(x)| + |T_{4\epsilon}\nu(x)| \\
	&= |T_\epsilon\nu(x)-T_{4\epsilon}\nu(x)| + |T(\chi_{\R^d\setminus B(x,4\epsilon)}\nu)(x)| \\
	&\leq |T_\epsilon\nu(x)-T_{4\epsilon}\nu(x)| + |T(\chi_{\R^d\setminus B(x,4\epsilon)}\nu)(x)-T(\chi_{\R^d\setminus B(x,4\epsilon)}\nu)(x')|+|T(\chi_{\R^d\setminus B(x,4\epsilon)}\nu)(x')|\\
	&\leq|T_\epsilon\nu(x)-T_{4\epsilon}\nu(x)| + |T(\chi_{\R^d\setminus B(x,4\epsilon)}\nu)(x)-T(\chi_{\R^d\setminus B(x,4\epsilon)}\nu)(x')|\\
	&\qquad+|T(\chi_{\R^d\setminus B(x,4\epsilon)}\nu)(x')-T_\epsilon\nu(x')|+|T_\epsilon\nu(x')|.
\end{aligned}
\end{equation*}

Now
\begin{equation*}
|T_\epsilon\nu(x)-T_{4\epsilon}\nu(x)| = \left|\int_{\epsilon\leq |x-y|<4\epsilon}k(x,y)d\nu(y)\right| \lesssim \int_{\epsilon < |x-y|\leq 4\epsilon}\frac{d|\nu|(y)}{|x-y|^n}\lesssim \frac{|\nu|[B(x,4\epsilon)]}{(4\epsilon)^n}\leq M^r_{\Phi_R}\nu(x).
\end{equation*}

In addition
\begin{equation*}
\begin{aligned}
|T(\chi_{\R^d\setminus B(x,4\epsilon)}\nu)(x)-T(\chi_{\R^d\setminus B(x,4\epsilon)}\nu)(x')| &= \left|\int_{|x-y|>4\epsilon}[k(x,y)-k(x',y)]d\nu(y)\right|\\
	&\lesssim \int_{|x-y|>4\epsilon}\frac{|x-x'|}{|x-y|^{n+1}}d|\nu|(y)\leq M^r_{\Phi_R}\nu(x),
\end{aligned}
\end{equation*}
where the last inequality is obtained by taking into account that $|x-x'|\leq \epsilon$ and splitting the domain of integration into annuli $\{2^k\epsilon < |x-y| \leq 2^{k+1}\epsilon\}$, $k=2,3,\dots$ Finally,

\begin{equation*}
\begin{aligned}
|T(\chi_{\R^d\setminus B(x,4\epsilon)}\nu)(x')-T_\epsilon\nu(x')| &= \left|\int_{|y-x|>4\epsilon}k(x',y)d\nu(y)-\int_{|y-x'|>\epsilon}k(x',y)d\nu(y)\right|\\
 &= \Bigg|\left(\int_{|y-x|>4\epsilon, |y-x'|\leq\epsilon}k(x',y)d\nu(y)+\int_{|y-x|>4\epsilon, |y-x'|>\epsilon}k(x',y)d\nu(y)\right)\\
 & \qquad- \left( \int_{|y-x'|>\epsilon, |y-x|>4\epsilon}k(x',y)d\nu(y) + \int_{|y-x'|>\epsilon, |y-x|\leq 4\epsilon}k(x',y)d\nu(y) \right)\Bigg|\\
 &= \left|\int_{|y-x|>4\epsilon, |y-x'|\leq\epsilon}k(x',y)d\nu(y)-\int_{|y-x'|>\epsilon, |y-x|\leq 4\epsilon}k(x',y)d\nu(y)\right|
\end{aligned}
\end{equation*}

Here, the first integral vanishes, since $|x-x'|<2\epsilon$ and $|y-x|\leq \epsilon$ imply that $|y-x|<3\epsilon$. Therefore,

\begin{equation*}
\begin{aligned}
|T(\chi_{\R^d\setminus B(x,4\epsilon)}\nu)(x')-T_\epsilon\nu(x')| &\leq \left|\int_{|y-x'|>\epsilon, |y-x|\leq 4\epsilon}k(x',y)d\nu(y)\right| \\
	&\lesssim \int_{|y-x'|>\epsilon, |y-x|\leq 4\epsilon}\frac{d|\nu|(y)}{|x'-y|^{n}}\\
	&\leq \frac{|\nu|[B(x,4\epsilon)]}{\epsilon^n} \lesssim M^r_{\Phi_R}\nu(x).
\end{aligned}
\end{equation*}

This completes the proof of \refeq{eq:APorCotlar} and, hence, of the lemma.

\end{proof}
\subsection{$L^2$-boundedness of $T_{\mu, \Phi_R}$}
$ $
\par\medskip
\begin{lemma}
\label{Le:TSigmaRAcotado}
Let $R\in \Top$ and consider the measure $\sigma_R = \theta_\mu(B_R)\H^n\lfloor_{\Gamma_R}$. Then, for $1<p<\infty$, $T_{\sigma_R, \Phi_R}$ is bounded from $L^p(\sigma_R)$ to $L^p(\chi_{B_0(R)}\mu)$, with norm bounded by $C_p\theta_\mu(B_R)$. Furthermore, $T_{\sigma_R, \Phi_R}$ is bounded from $L^1(\sigma_R)$ to $L^{1,\infty}(\chi_{B_0(R)}\mu)$, with norm bounded by $C\theta_\mu(B_R)$.
\end{lemma}

\begin{proof}
First of all, we observe that the maximal operator $M^r_{\sigma_R, \Phi_R}$ is bounded from $L^\infty(\sigma_R)$ to $L^\infty(\chi_{B_0(R)}\mu)$ with norm bounded by $C\theta_\mu(B_R)$. Indeed, if $f\in L^{\infty}(\sigma_R)$, and $x\in B_0(R)$

\begin{equation*}
\begin{aligned}
M^r_{\sigma_R, \Phi_R} f(x) &= \sup_{r\geq \Phi_R(x)}\frac{1}{r^n}\int_{B(x,r)\cap B_0(R)}|f|d\mu \leq ||f||_{L^{\infty}(\sigma_R)}\sup_{r\geq \Phi_R(x)}\frac{\mu[B(x,r)\cap B_0(R)]}{r^n}\\
	&\lesssim \theta_\mu(B_R)||f||_{L^{\infty}(\sigma_R)},
\end{aligned}
\end{equation*}
by $(4)$ in lemma \ref{Le:PropiedadesPhiR}. Therefore,
\begin{equation*}
||M^r_{\sigma_R,\Phi_R} f||_{L^{\infty}(\chi_{B_0(R)}\mu)} \lesssim \theta_\mu(B_R)||f||_{L^{\infty}(\sigma_R)},
\end{equation*}
as claimed.
\par\medskip
Now, let us check that $M^r_{\sigma_R,\Phi_R}$ is bounded from $L^1(\sigma_R)$ to $L^{1,\infty}(\chi_{B_0(R)}\mu)$ with norm bounded by $C\theta_\mu(B_R)$. In fact, we will prove a slightly stronger result, as we will deal with a non-centered version of $M^r_{\sigma_R,\Phi_R}$, which will be useful for technical reasons. Define, for $f\in L^1(\sigma_R)$ and $x\in\R^d$,
\begin{equation*}
N^r_{\sigma_R, \Phi_R}f(x) = \sup\frac{1}{r(B)^n}\int_{B}|f|d\sigma_R,
\end{equation*}
where the supremum is taken over all balls $B$ with $x\in B$ and such that $\mu(5B)\leq C_1\theta_\mu(B_R)(5r(B))^n$, where $C_1$ is the same constant that appears in $(4)$ of lemma \ref{Le:PropiedadesPhiR}. Clearly,
\begin{equation*}
M^r_{\sigma_R,\Phi_R}f(x) \leq N^r_{\sigma_R,\Phi_R}f(x),
\end{equation*}
so the weak $(1,1)$ inequality for $M^r_{\sigma_R,\Phi_R}$ will follow from that for $N^r_{\sigma_R, \Phi_R}$.
\par\medskip
Let $f\in L^1(\sigma_R)$, $\lambda>0$, and consider
\begin{equation*}
\Omega_\lambda = \{x\in B_0(R)\colon N^r_{\sigma_R, \Phi_R}f(x) >\lambda\}
\end{equation*}
By definition of $N^r_{\sigma_R, \Phi_R}$, for every $x\in\Omega_\lambda$, there exists a ball $B_x$ containing $x$ with $\mu(5B_x)\leq C_1\theta_\mu(B_R)(5r(B))^n$ and such that
\begin{equation*}
\frac{1}{r(B_x)^n}\int_{B_x}|f|d\sigma_R>\lambda,
\end{equation*}
which is equivalent to
\begin{equation}
\label{eq:Debil11}
r(B_x)^n < \frac{1}{\lambda}\int_{B_x}|f|d\sigma_R.
\end{equation}
Now, applying the $5r$-covering theorem, we may extract a countable and disjoint subfamily $\{B_i\}$ of $\{B_x\}_{x\in\Omega_\lambda}$ such that the balls $\{5B_i\}$ cover $\Omega_\lambda$.  Then, we have 
\begin{equation}
\label{eq:Weak11NR}
\begin{aligned}
\mu(\Omega_\lambda) &\leq \sum_i \mu(5B_i) \leq \sum_i C_1\theta_\mu(B_R)(5r(B_{i}))^n \lesssim \theta_\mu(B_R)\sum_i r(B_{i})^n\\
	&\leq \theta_\mu(B_R) \sum_i \frac{1}{\lambda}\int_{B_{i}}|f|d\sigma_R \leq \frac{\theta_\mu(B_R)}{\lambda}\int_{\Omega_\lambda}|f|d\sigma_R \leq \frac{\theta_\mu(B_R)}{\lambda}||f||_{L^1(\sigma_R)},
\end{aligned}
\end{equation}
which proves that $N^r_{\sigma_R,\Phi_R}$ (and also  $M^r_{\sigma_R,\Phi_R}$) is bounded from $L^1(\sigma_R)$ to $L^{1,\infty}(\chi_{B_0(R)}\mu)$ with norm bounded by $C\theta_\mu(B_R)$. Then, Marcinkiewicz's Interpolation Theorem applies and so, for $1<p<\infty$ $M^r_{\sigma_R,\Phi_R}$ is bounded from $L^p(\sigma_R)$ to $L^p(\chi_{B_0(R)}\mu)$ with norm bounded by $C_p\theta_\mu(B_R)$

\par\medskip

Notice that \refeq{eq:cotlar} in Lemma \ref{Le:Cotlar} can be restated as
\begin{equation}
\label{eq:ReCotlar}
T_{\sigma_R,\Phi_R,*}f(x) \leq C_s [M^r_{\sigma_R,\Phi_R}((T_{\H^n\lfloor_{\Gamma_R}}f)^s)(x)^{\frac{1}{s}}+M^r_{\sigma_R,\Phi_R}f(x)].
\end{equation}
Then, taking $s=1$ and using the $L^p(\sigma_R)\rightarrow L^p(\chi_{B_0(R)}\mu)$-boundedness of $M^r_{\sigma_R,\Phi_R}$, we obtain that $T_{\sigma_R,\Phi_R,*}$ is bounded from $L^p(\sigma_R)$ to $L^p(\chi_{B_0(R)}\mu)$ with norm bounded by $C_p\theta_\mu(B_R)$.
\par\medskip
To deal with the weak $(1,1)$ case, we will need to work a little harder. Going back to  \refeq{eq:ReCotlar}, with $s=\frac{1}{2}$, we get that for $f\in L^1(\sigma_R)$,
\begin{equation*}
T_{\sigma_R,\Phi_R,*}f(x) \leq C[M^r_{\sigma_R,\Phi_R}((T_{\H^n\lfloor_{\Gamma_R}}f)^\frac{1}{2})(x)^2+M^r_{\sigma_R,\Phi_R}f(x)]
\end{equation*}
and so, for $\lambda>0$,

\begin{equation*}
\begin{aligned}
\mu(\{x\in B_0(R)\colon T_{\sigma_R,\Phi_R,*}f(x)>\lambda\}) &\leq \mu\left(\left\{x\in B_0(R)\colon M^r_{\sigma_R,\Phi_R}((T_{\H^n\lfloor_{\Gamma_R}}f)^\frac{1}{2})(x)^2 > \frac{\lambda}{2C} \right\}\right) \\
	& \qquad + \mu\left(\left\{x\in B_0(R)\colon M^r_{\sigma_R,\Phi_R}f(x) > \frac{\lambda}{2C} \right\}\right) \\
	&\leq \mu\left(\left\{x\in B_0(R)\colon M^r_{\sigma_R,\Phi_R}((T_{\sigma_R}f)^\frac{1}{2})(x) > \left(\frac{\lambda}{2C}\right)^{\frac{1}{2}}\theta_\mu (B_R)^{\frac{1}{2}} \right\}\right) \\
	& \qquad + \mu\left(\left\{x\in B_0(R)\colon M^r_{\sigma_R,\Phi_R}f(x) > \frac{\lambda}{2C} \right\}\right)
\end{aligned}
\end{equation*}

Here, the second term is bounded by $C\frac{\theta_\mu(B_R)}{\lambda}||f||_{L^1(\sigma_R)}$ because of the weak $(1,1)$-inequality for $M^r_{\sigma_R,\Phi_R}$. To deal with the first term, we will use the weak $(1,1)$-inequality \refeq{eq:Weak11NR} for $N^r_{\sigma_R,\Phi_R}$. Denote
\begin{equation*}
\Omega = \left\{x\in B_0(R)\colon N^r_{\sigma_R,\Phi_R}((T_{\sigma_R}f)^{\frac{1}{2}})(x)>\left(\frac{\lambda}{2C}\right)^{\frac{1}{2}}\theta_\mu (B_R)^{\frac{1}{2}}\right\}
\end{equation*}
so that
\begin{equation*}
\begin{aligned}
\mu\left(\left\{x\in B_0(R)\colon M^r_{\sigma_R,\Phi_R}((T_{\sigma_R}f)^\frac{1}{2})(x) > \left(\frac{\lambda}{2C}\right)^{\frac{1}{2}}\theta_\mu (B_R)^{\frac{1}{2}} \right\}\right) &\leq \mu(\Omega) \lesssim \frac{\theta_\mu(B_R)}{\lambda^\frac{1}{2}\theta_\mu(B_R)^{\frac{1}{2}}}\int_{\Omega}|T_{\sigma_R}f|^{\frac{1}{2}}d\mu \\
	&\lesssim \frac{\theta_\mu(B_R)^{\frac{1}{2}}}{\lambda^{\frac{1}{2}}}\mu(\Omega)^{\frac{1}{2}}||T_{\sigma_R}f||^{\frac{1}{2}}_{L^{1,\infty}(\mu)}\\
	&=\mu(\Omega)^{\frac{1}{2}}\frac{1}{\lambda^{\frac{1}{2}}}||T_{\sigma_R}f||_{L^{1,\infty}(\sigma_R)}^{\frac{1}{2}},
\end{aligned}
\end{equation*}
which implies that $\mu(\Omega)\lesssim \frac{1}{\lambda}||T_{\sigma_R}f||_{L^{1,\infty}(\sigma_R)}$, and therefore
\begin{equation*}
\mu\left(\left\{x\in B_0(R)\colon M^r_{\sigma_R,\Phi_R}((T_{\sigma_R}f)^\frac{1}{2})(x) > \frac{\lambda^{\frac{1}{2}}}{\sqrt{2C}}\theta_\mu(B_R) \right\}\right) \lesssim \frac{1}{\lambda}||T_{\sigma_R}f||_{L^{1,\infty}(\sigma_R)} \lesssim \frac{\theta_\mu(B_R)}{\lambda}||f||_{L^1(\sigma_R)},
\end{equation*}
where we used the fact that $T_{\sigma_R}$ is bounded from $L^1(\sigma_R)$ to $L^{1,\infty}(\sigma_R)$ with norm bounded by $C\theta_\mu(B_R)$. This completes the proof of the lemma.

\end{proof}

We recall here a lemma that is also used at \cite{Tolsa-pubMat} that will be useful. Its proof is based on the combined use of both Marcinkiewicz's and Riesz-Thorin's Interpolation Theorems.

\begin{lemma}
Let $\tau$ be a Radon measure in $\R^d$ and let $T$ be a linear operator that is bounded in $L^2(\tau)$ with norm $N_2$. Suppose further that both $T$ and its adjoint $T^*$ are bounded from $L^1(\tau)$ to $L^{1,\infty}(\tau)$ with norm bounded by $N_1$. Then $N_2\leq cN_1$, where $c$ is an absolute constant.
\end{lemma}

\begin{lemma}
\label{Le:TPhiRMaxAcotadoL2}
$T_{\mu,\Phi_R}$ is bounded on $L^2(\chi_{B_0(R)}\mu)$ with norm bounded by $C\theta_\mu(B_R)$.
\end{lemma}
\begin{proof}
Since $T_{\mu,\Phi_R}$ is antisymmetric, by the previous lemma, we can limit ourselves to prove that it is bounded from $L^1(\chi_{B_0(R)}\mu)$ to $L^{1,\infty}(\chi_{B_0(R)}\mu)$ with norm bounded by $C\theta_\mu(B_R)$.
\par\medskip
Let $f\in L^1(\chi_{B_0(R)})$ and denote $\Reg(R) = \{Q_i\}_{i=1}^{\infty}$, where we assume that the side-lengths $\l(Q_i)$ are non-increasing. Arguing as in $(4)$ of lemma $\ref{Le:PropCubosRegulares}$, it is easy to check that every cell $Q_i$ is contained in a cell $Q_i'$ such that $\theta_\mu(Q_i')\lesssim \theta_\mu(B_R)$, $\dmu(Q_i,Q_i')\lesssim \theta_\mu(B_R)$, $Q_i'\cap\Gamma_R\neq \emptyset$ and $\H^n(Q_i'\cap\Gamma_R)\approx \l(Q_i')^n$.

\par\medskip
Set
\begin{equation*}
g = f \chi_{B_0(R)\setminus\bigcup_i Q_i}, \;\; b = \sum_i f\chi_{Q_i}
\end{equation*}
so that $f=g+b$. Since $B_0(R)\setminus\bigcup_i Q_i\subset \Good(R)$ and this is contained in $\Gamma_R$ (up to a set of $\mu$-measure zero), by the Radon-Nikodym theorem we obtain that
\begin{equation*}
\mu\lfloor_{B_0(R)\setminus\bigcup_i Q_i} = \eta\H^n_{\Gamma_R},
\end{equation*}
where $\eta$ is some function with $0\leq \eta\leq C\theta_\mu[B_0(R)]\lesssim \theta_\mu(B_R)$. Then, by lemma \ref{Le:TSigmaRAcotado}, we have that, for $\lambda>0$,
\begin{equation}
\label{eq:78}
\begin{aligned}
\mu(\{x\in B_0(R)\colon |T_{\mu, \Phi_R}g(x)|>\lambda\}) &= \mu(\{x\in B_0(R)\colon |T_{\H^n_{\Gamma_R},\Phi_R}(g\eta)(x)|>\lambda\}) \\
	&= \mu\left(\left\{x\in B_0(R)\colon |T_{\sigma_R,\Phi_R}(g\eta)(x)|>\theta_\mu(B_R)\lambda\right\}\right) \\
	&\lesssim \frac{1}{\lambda}||g\eta||_{L^1(\sigma_R)} = \frac{\theta_\mu(B_R)}{\lambda}||g\eta||_{L^1(\H^n\lfloor_{\Gamma_R})} = \frac{\theta_\mu(B_R)}{\lambda}||f||_{L^1(\mu)}
\end{aligned}
\end{equation}

Now, to deal with $T_{\mu,\Phi_R}b$, we define, for every $i\geq 1$
\begin{equation*}
\gamma_i(x) = \left(\frac{1}{\H^n(B_{Q_i'}\cap \Gamma_R)}\int_{Q_i}fd\mu \right)\chi_{B_{Q_i'}\cap \Gamma_R}(x), \;\; \nu_i = (f\chi_{Q_i})\mu - \gamma_i \H^n_{\Gamma_R},
\end{equation*}
so that $\nu_i$ is supported on $B_{Q_i'}$ and satisfies $\int d\nu_i = 0$ , and we write
\begin{equation*}
b\mu = \sum_i \nu_i + \sum_i \gamma_i\H^n_{\Gamma_R}
\end{equation*}
so that
\begin{equation*}
T_{\mu,\Phi_R}b = T_{\Phi_R}(b\mu) = T_{\Phi_R}\left(\sum_i \nu_i\right) + T_{\Phi_R}\left(\sum_i\gamma_i \H^n_{\Gamma_R}\right).
\end{equation*}

Now, again by lemma \ref{Le:TSigmaRAcotado}, we get
\begin{equation}
\label{eq:82}
\begin{aligned}
\mu\left(\left\{x\in B_0(R)\colon \left|T_{\Phi_R}\left(\sum_i\gamma_i \H^n_{\Gamma_R}\right)(x)\right|>\lambda\right\}\right) &= \mu\left(\left\{x\in B_0(R)\colon \left|T_{\Phi_R,\sigma_R}\left(\sum_i\gamma_i \right)(x)\right|>\theta_\mu(B_R)\lambda\right\}\right)\\
	&\lesssim \frac{1}{\lambda}\left|\left|\sum_{i}\gamma_i\right|\right|_{L^1(\sigma_R)}\leq \frac{\theta_\mu(B_R)}{\lambda}\sum_i \int|\gamma_i|d\H^n_{\Gamma_R}\\
	&\leq \frac{\theta_\mu(B_R)}{\lambda}||f||_{L^1(\mu)}.
\end{aligned}
\end{equation}

Finally, to deal with the term $\displaystyle{T_{\Phi_R}\left(\sum_i \nu_i\right)}$, we apply Chebishev's inequality to get

\begin{equation}
\label{eq:83}
\begin{aligned}
\mu\left(\left\{x\in B_0(R)\colon \left|T_{\Phi_R}\left(\sum_i \nu_i\right)(x)\right|>\lambda\right\}\right) &\leq \frac{1}{\lambda}\int_{B_0(R)}\left|T_{\Phi_R}\left(\sum_i\nu_i\right)\right|d\mu\\
	&= \frac{1}{\lambda}\left(\sum_{i}\int_{2B_{Q_i'}}|T_{\Phi_R}\nu_i|d\mu + \int_{B_0(R)\setminus 2B_{Q_i'}}|T_{\Phi_R}\nu_i|d\mu\right)
\end{aligned}
\end{equation}

Now, since $\int d\nu_i = 0$, for $x\not\in 2B_{Q_i'}$ we have

\begin{equation*}
\begin{aligned}
|T_{\Phi_R}\nu_i(x)| &= \left|\int_{B_{Q_i'}} k_{\Phi_R}(x,y)d\nu_i(y)\right| = \left|\int_{B_{Q_i'}}[k_{\Phi_R}(x,y)-k_{\Phi_R}(x,z_{Q_i'})]d\nu_i(y)\right|\\
	&\lesssim \int_{B_{Q_i'}}\frac{|y-z_{Q_i'}|}{|x-z_{Q_i'}|^{n+1}}d|\nu_i|(y)\lesssim \frac{\l(Q_i')||\nu_i||}{|x-z_{Q_i'}|^{n+1}}
\end{aligned}
\end{equation*}
and so
\begin{equation}
\label{eq:85}
\int_{\R^d\setminus 2B_{Q_i'}}|T_{\Phi_R}\nu_i|d\mu \lesssim \int_{B_0(R)\setminus 2B_{Q_i'}}\frac{\l(Q_i')||\nu_i||}{|x-z_{Q_i'}|^{n+1}}d\mu\lesssim \theta_\mu(B_R)||\nu_i||\lesssim \theta_\mu(B_R)\int_{Q_i}|f|d\mu.
\end{equation}

On the other hand,
\begin{equation*}
\begin{aligned}
\int_{2B_{Q_i'}}|T_{\Phi_R}\nu_i|d\mu &\leq \int_{2B_{Q_i'}}|T_{\Phi_R}((f\chi_{Q_i})\mu)|d\mu + \int_{2B_{Q_i'}}|T_{\Phi_R}( \gamma_i \H^n_{\Gamma_R})|d\mu\\
	&\leq \int_{Q_i}|T_{\Phi_R}((f\chi_{Q_i})\mu)|d\mu + \int_{2B_{Q_i'}\setminus Q_i}|T_{\Phi_R}((f\chi_{Q_i})\mu)|d\mu + \int_{2B_{Q_i'}}|T_{\Phi_R}( \gamma_i \H^n_{\Gamma_R})|d\mu\\
	&=\I_1 + \I_2 + \I_3.
\end{aligned}
\end{equation*}

Now, to bound $\I_1$ we use the fact that for all $x\in Q_i$, $\Phi_R(x)\geq \l(Q_i)$, by $(3)$ in lemma \ref{Le:PropiedadesPhiR}, and so $|k_{\Phi_R}(x,y)|\lesssim \l(Q_i)$ for all $x,y\in Q_i$. Hence,
\begin{equation*}
|T_{\Phi_R}((f\chi_{Q_i})\mu)(x)|\lesssim \frac{1}{\l(Q_i)^n}\int_{Q_i}|f|d\mu
\end{equation*}
and so
\begin{equation*}
\I_1 \lesssim \frac{\mu(Q_i)}{\l(Q_i)^n}\int_{Q_i}|f|d\mu \lesssim \theta_\mu(B_R)\int_{Q_i}|f|d\mu,
\end{equation*}
by $(4)$ in lemma \ref{Le:PropiedadesPhiR}.

\par\medskip

To bound $\I_2$, we observe that for $x\in 2B_{Q_i'}\setminus Q_i$,
\begin{equation*}
|T_{\Phi_R}((\chi_{Q_i}f)\mu)(x)| = \left|\int_{Q_i}k_{\Phi_R}(x,y)f(y)d\mu(y)\right|\lesssim \frac{1}{|x-z_{Q_i}|^n}\int_{Q_i}|f|d\mu
\end{equation*}
and so
\begin{equation*}
\begin{aligned}
\I_2 &= \int_{2B_{Q_i'}\setminus Q_i}|T_{\Phi_R}((f\chi_{Q_i})\mu)|d\mu \lesssim \int_{Q_i}|f|d\mu\int_{2B_{Q_i'}\setminus Q_i}\frac{1}{|x-z_{Q_i}|^n}d\mu(x) \\
	&= \dmu(Q_i,Q_i')\int_{Q_i}|f|d\mu \lesssim \theta_\mu(B_R)\int_{Q_i}|f|d\mu.
\end{aligned}
\end{equation*}

Finally, by lemma \ref{Le:TSigmaRAcotado}
\begin{equation*}
\begin{aligned}
\I_3 &= \int_{2B_{Q_i'}}|T_{\Phi_R}( \gamma_i \H^n_{\Gamma_R})|d\mu \leq \mu(2B_{Q_i'})^{\frac{1}{2}}\left(\int_{2B_{Q_i'}}|T_{\Phi_R}( \gamma_i \H^n_{\Gamma_R})|^2d\mu\right)^{\frac{1}{2}} \\
	 &\leq \frac{1}{\theta_\mu(B_R)}\mu(2B_{Q_i'})^{\frac{1}{2}}\left(\int|T_{\Phi_R}( \gamma_i \sigma_R)|^2d\mu\right)^{\frac{1}{2}}\lesssim \mu(2B_{Q_i'})^\frac{1}{2}||\gamma_i||_{L^2(\sigma_R)}\\
	 &\leq \mu(Q_i')^{\frac{1}{2}}\theta_\mu(B_R)^{\frac{1}{2}}\frac{1}{\H^n(Q_i'\cap\Gamma_R)}\int_{Q_i}|f|d\mu \lesssim \theta_\mu(B_R)\int_{Q_i}|f|d\mu.
\end{aligned}
\end{equation*}

Gathering the estimates for $\I_1$, $\I_2$ and $\I_3$, we obtain
\begin{equation*}
\int_{2B_{Q_i'}}|T_{\Phi_R}\nu_i|d\mu\lesssim \theta_\mu(B_R)\int_{Q_i}|f|d\mu,
\end{equation*}
and so, going back to \refeq{eq:83} and also taking into account \refeq{eq:85}, we obtain
\begin{equation*}
\mu\left(\left\{x\in B_0(R)\colon \left|T_{\Phi_R}\left(\sum_i \nu_i\right)(x)\right|>\lambda\right\}\right) \lesssim \frac{1}{\lambda}\int|f|d\mu
\end{equation*}

This, together with \refeq{eq:78} and \refeq{eq:82}, imply the weak $(1,1)$ inequality
\begin{equation*}
\mu\left(\left\{x\in B_0(R)\colon |T_{\mu,\Phi_R}f(x)|>\lambda\right\}\right) \lesssim \frac{\theta_\mu(B_R)}{\lambda}||f||_{L^1(\mu)}
\end{equation*}
that we were looking for.
\end{proof}

\subsection{$L^2$-boundedness of $T_{\Phi_R, \mu, *}$.}
\begin{lemma}
\label{Le:TPhiRMuAcotado}
For $R\in\Top$, $T_{\Phi_R,\mu,*}$ is bounded in $L^2(\chi_{B_0(R)}\mu)$ with norm bounded by $c\theta_\mu(B_R)$.
\end{lemma}

\begin{proof}
This is a direct consecuence of Theorem \ref{Th:CotlarNTV} and lemma \ref{Le:TPhiRMaxAcotadoL2}, taking $S=T_{\Phi_R}$, $\sigma=\chi_{B_0(R)}\mu$ and $C_0 \approx\theta_\mu(B_R)$.
\end{proof}
\par\medskip

With all these tools at hand, we can prove lemma \ref{Le:Diagonal}. Indeed, given $R\in\Top$, by lemmas \ref{Le:DesigualdadPuntual} and \ref{Le:TPhiRMuAcotado} we have
\begin{equation*}
||K_R\mu||_{L^2(\mu)}\leq ||T_{\Phi_R,*}(\chi_{B_0(R)}\mu)||_{L^2(\chi_R \mu)} + c\theta_\mu(B_R)\mu(R)^{\frac{1}{2}} \lesssim \theta_\mu(B_R)\mu(R)^{\frac{1}{2}},
\end{equation*}
and the desired conclusion follows after squaring both sides and summing over $R\in\Top$.

\section{The estimate of $\sum_{R, R'\in \Top, R\neq R'}\langle K_R\mu, K_{R'}\mu\rangle_\mu$}

Given $R,R'\in\Top$, $R\neq R'$, $\langle K_{R}\mu,  K_{R'}\mu\rangle_\mu = 0$ unless $R\cap R'\neq \emptyset$. Then,

\begin{equation*}
\sum_{R, R'\in \Top, R\neq R'}\langle K_R\mu, K_{R'}\mu\rangle_\mu = 2 \sum_{Q, R\in \Top, Q\subsetneq R}\langle K_Q\mu, K_R\mu\rangle_\mu
\end{equation*}

Arguing as in \cite{Tolsa-pubMat}, we can guess that bounding this sum would be relatively easy if 
\begin{equation*}
\int_{Q}K_Q\mu = 0,
\end{equation*}
but this is, in general, not the case. Indeed,
\begin{equation*}
K_Q\mu = \sum_{M\in\Tree(R)}T_M\mu = \sum_{M\in\Tree(R)}\chi_M T_{J(M)}\mu,
\end{equation*}
and while it is true that for all $M\in\Tree(R)$
\begin{equation*}
\int_{M}T_{J(M)}(\chi_M\mu)d\mu = 0
\end{equation*}
by antisimmetry, this does not imply that
\begin{equation*}
\int_{M}T_{J(M)}\mu=0
\end{equation*}
and so
\begin{equation*}
\int_Q K_Q\mu d\mu = 0
\end{equation*}
will not be true in general. Still, the fact that
\begin{equation*}
\int_M T_i(\chi_M\mu)d\mu = 0
\end{equation*}
for all $i\geq 0$ and all $M\in\D$ will be useful, as we will see in the proof of lemma \ref{Le:AcotacionND1}.
\par\medskip
We have
\begin{equation*}
\begin{aligned}
\sum_{Q, R\in \Top, Q\subsetneq R}\langle K_Q\mu, K_R\mu\rangle_\mu &= \sum_{R\in\Top}\sum_{P\in\Stop(R)}\sum_{Q\in\Top, Q\subset P} \langle K_Q\mu, K_R\mu\rangle_\mu \\
	&= \sum_{R\in\Top}\sum_{P\in\Stop(R)}\sum_{Q\in\Top, Q\subset P} \sum_{Q'\in\Tree(Q)}\langle T_{Q'}\mu, K_R\mu\rangle_\mu \\
	&= \sum_{R\in\Top}\sum_{P\in\Stop(R)}\sum_{Q\in\D(P)}\langle T_Q\mu,K_R\mu\rangle_\mu \\
	&= \sum_{R\in\Top}\sum_{P\in\Stop(R)}\sum_{i=J(P)}^{\infty}\sum_{Q\in\D_i(P)}\langle \chi_Q T_i\mu,K_R\mu\rangle_\mu \\
	&= \sum_{R\in\Top}\sum_{P\in\Stop(R)}\sum_{i=J(P)}^{\infty}\langle \chi_P T_i\mu,K_R\mu\rangle_\mu \\
\end{aligned}
\end{equation*}
\par\medskip
Now, fixed $R\in\Top$, $P\in\Stop(R)$ and $i\geq J(P)$, we define $m(J(P),i)$ as some intermediate number between $J(P)$ and $i$ (for example, the integer part of the arithmetic mean of $J(P)$ and $i$), and we decompose 
\begin{equation*}
P = \bigcup_{S\in\D_{m(J(P),i)}\colon S\subset P} S
\end{equation*}
so that
\begin{equation*}
\begin{aligned}
\sum_{Q, R\in \Top, Q\subsetneq R}\langle K_Q\mu, K_R\mu\rangle_\mu & = \sum_{R\in\Top}\sum_{P\in\Stop(R)}\sum_{i=J(P)}^{\infty}\langle \chi_P T_i\mu,K_R\mu\rangle_\mu \\
	&= \sum_{R\in\Top}\sum_{P\in\Stop(R)}\sum_{i=J(P)}^{\infty}\sum_{S\in\D_{m(J(P),i)}}\langle \chi_S T_i\mu,K_R\mu\rangle_\mu \\
	&=\sum_{R\in\Top}\sum_{P\in\Stop(R)}\sum_{i=J(P)}^{\infty}\sum_{S\in\D_{m(J(P),i)}}\langle \chi_S T_i(\chi_S\mu),K_R\mu\rangle_\mu \\
	& \qquad + \sum_{R\in\Top}\sum_{P\in\Stop(R)}\sum_{i=J(P)}^{\infty}\sum_{S\in\D_{m(J(P),i)}}\langle \chi_S T_i(\chi_{\R^d\setminus S}\mu),K_R\mu\rangle_\mu := \ND_1 + \ND_2
\end{aligned}
\end{equation*}

\subsection{The estimate of $\ND_1$.}

\begin{lemma}
\label{Le:AcotacionND1}
\begin{equation*}
\ND_1 \lesssim  \sum_{R\in\Top}\theta_\mu(B_R)^2\mu(R)
\end{equation*}
\end{lemma}

\begin{proof}
Recall that
\begin{equation*}
\ND_1 = \sum_{R\in\Top}\sum_{P\in\Stop(R)}\sum_{i=J(P)}^{\infty}\sum_{S\in\D_{m(J(P),i)}}\langle \chi_S T_i(\chi_S\mu),K_R\mu\rangle_\mu
\end{equation*}

Fix $R\in\Top$, $P\in\Stop(R)$, $i\geq J(P)$ and $S\in\D_m(J(P),i)$. Since
\begin{equation*}
\int_S T_i(\chi_S\mu)d\mu = 0,
\end{equation*}
we have
\begin{equation*}
\langle \chi_S T_i(\chi_S\mu),K_R\mu\rangle_\mu = \int_S T_i(\chi_S\mu) K_R\mu d\mu = \int_S T_i(X_S\mu)[K_R\mu - K_R\mu(z_S)]d\mu.
\end{equation*}

Now, given $x\in S$, since $S\subset P$ and $P\in\Stop(R)$, we have that the cells from $\Tree(R)$ that contain $Q$ are the chain in $\D$ that starts in the parent of $P$ and ends in $R$. Therefore,

\begin{equation*}
\begin{aligned}	
K_R\mu(x) &= \sum_{Q\in\Tree(R)\colon x\in Q} T_Q\mu(x) \\
	&= \sum_{j\in J(R)}^{J(P)-1}T_j\mu(x) \\
	&= \int\left(\sum_{j\in J(R)}^{J(P)-1}\varphi_j(x-y)\right)k(x,y)d\mu(y) \\
	&= \int \left[ \psi_{J(R)}(x-y) - \psi_{J(P)}(x-y)\right] k(x,y)d\mu(y)
\end{aligned}
\end{equation*}

If we denote 
\begin{equation*}
\zeta_{R,P}(x,y) = \left[ \psi_{J(R)}(x-y) - \psi_{J(P)}(x-y)\right] k(x,y)
\end{equation*}
it is easy to check that for $x,x'\in S$ we have
\begin{equation*}
|\zeta_{R,P}(x,y)-\zeta_{R,P}(x',y)|\lesssim \frac{|x-x'|}{(\l(P)+|x-y|)^{n+1}}.
\end{equation*}

Therefore, for $x\in S$,
\begin{equation*}
\begin{aligned}
|K_R\mu(x)-K_R\mu(z_S)| &\lesssim \int_{\dist(y,P)\leq 0.01A_0^{-J(R)}} \frac{|x-z_S|}{(\l(P)+|x-y|)^{n+1}}d\mu(y) \\
	&\lesssim \frac{\l(S)}{\l(P)}\theta_\mu(B_R),
\end{aligned}
\end{equation*}
where the last inequality follows from \refeq{Eq:CrecLineal}, and so
\begin{equation*}
\begin{aligned}
|\langle \chi_ST_i(\chi_S\mu),K_R\mu\rangle_\mu| &\lesssim \frac{\l(S)}{\l(P)}\theta_\mu(B_R)\int_S |T_i(\chi_S\mu)|d\mu \\
	&= \frac{\l_{m(J(P),i)}}{\l_{J(P)}}\theta_\mu(B_R)\int_S |T_i(\chi_S\mu)|d\mu \\
	&\approx A_0^{\frac{J(P)-i}{2}}\theta_\mu(B_R)\int_S |T_i(\chi_S\mu)|d\mu.
\end{aligned}
\end{equation*}

Now, for $x\in S$,
\begin{equation*}
\begin{aligned}
|T_i(\chi_S\mu)(x)| &= \left|\int_{y\in S}\varphi_i(x-y)k(x,y)d\mu(y)\right| \\
	&= \left|\int_{y\in S, \, 0.001 A_0^{-i-1}<|x-y|<0.01A_0^{-i}}\varphi_i(x-y)k(x,y)d\mu(y)\right|\\
	&\lesssim \int_{y\in S, \, 0.001 A_0^{-i-1}<|x-y|<0.01A_0^{-i}}\frac{d\mu(y)}{|x-y|^{n}}\\
	&\lesssim \frac{\mu[B(x,0.01A_0^{-i})]}{A_0^{-ni}}:=\theta_{\mu,i}(x)
\end{aligned}
\end{equation*}
and so
\begin{equation*}
|\langle \chi_ST_i(\chi_S\mu),K_R\mu\rangle_\mu| \lesssim A_0^{\frac{J(P)-i}{2}}\theta_\mu(B_R)\int_S\theta_{\mu,i}(x)d\mu(x).
\end{equation*}
Therefore,
\begin{equation*}
\begin{aligned}
\ND_1 &\leq \sum_{R\in\Top}\sum_{P\in\Stop(R)}\sum_{i=J(P)}^{\infty}\sum_{S\in\D_{m(J(P),i)}} |\langle \chi_ST_i(\chi_S\mu),K_R\mu\rangle_\mu| \\
	&\lesssim \sum_{R\in\Top}\sum_{P\in\Stop(R)}\sum_{i=J(P)}^{\infty}\sum_{S\in\D_{m(J(P),i)}} A_0^{\frac{J(P)-i}{2}}\theta_\mu(B_R)\int_S\theta_{\mu,i}(x)d\mu(x)\\
	&\lesssim \sum_{R\in\Top}\theta_\mu(B_R)\sum_{P\in\Stop(R)}A_0^{\frac{J(P)}{2}}\sum_{i=J(P)}^{\infty}A_0^{-\frac{i}{2}}\int_P\theta_{\mu,i}(x)d\mu(x)\\
	&= \sum_{R\in\Top}\theta_\mu(B_R)\sum_{P\in\Stop(R)}A_0^{\frac{J(P)}{2}}\sum_{i=J(P)}^{\infty}A_0^{-\frac{i}{2}}\sum_{P'\in\D_i\colon P'\subset P}\int_{P'}\theta_{\mu,i}(x)d\mu(x)\\
	&\lesssim \sum_{R\in\Top}\theta_\mu(B_R)\sum_{P\in\Stop(R)}A_0^{\frac{J(P)}{2}}\sum_{P'\in\D(P)}A_0^{-\frac{J(P')}{2}}\theta_\mu[1.01B_{P'}]\mu(P'),
\end{aligned}
\end{equation*}
We reorganize the previous sum, to obtain 
\begin{equation}
\label{eq:142}
\begin{aligned}
\ND_1 &\lesssim \sum_{R\in\Top}\theta_\mu(B_R)\sum_{P\in\Stop(R)}A_0^{\frac{J(P)}{2}}\sum_{P''\in \Top\colon P''\subset P}\sum_{P'\in\Tree(P'')}A_0^{-\frac{J(P')}{2}}\theta_\mu[1.01B_{P'}]\mu(P')\\
\end{aligned}
\end{equation}
and from the fact that $P'\in\Tree(P'')$, we obtain that $\theta_\mu(1.01B_{P'})\lesssim \theta_\mu(B_{P''})$, so
\begin{equation}
\label{eq:143}
\begin{aligned}
\ND_1 &\lesssim \sum_{R\in\Top}\theta_\mu(B_R)\sum_{P\in\Stop(R)}A_0^{\frac{J(P)}{2}}\sum_{P''\in \Top\colon P''\subset P}\theta_\mu(B_{P''})\sum_{P'\in\Tree(P'')}A_0^{-\frac{J(P')}{2}}\mu(P')\\
	&\lesssim \sum_{R\in\Top}\theta_\mu(B_R)\sum_{P\in\Stop(R)}A_0^{\frac{J(P)}{2}}\sum_{P''\in \Top\colon P''\subset P}\theta_\mu(B_{P''})A_0^{-\frac{J(P'')}{2}}\mu(P'')\\
	&=\sum_{R\in\Top} \theta_\mu(B_R)\sum_{P''\in\Top\colon P''\subsetneq R} A_0^{\frac{J(R_{P''})-J(P'')}{2}}\theta_\mu(B_{P''})\mu(P'')
\end{aligned}
\end{equation}
where, given $R,P''\in\Top$ with $P''\subsetneq R$, $R_{P''}$ is the cell from $\Stop(R)$ that contains $P''$. To deal with this sum, we need to organize it in trees. To do so, define $\Stop^1(R) = \Stop(R)$ and, for $k>1$,
\begin{equation*}
\Stop^k(R) = \{Q\in\D(R)\colon \text{there exists }Q'\in\Stop^{k-1}(R)\text{ with }Q\in\Stop(Q')\}
\end{equation*}
so that
\begin{equation*}
\{P\in\Top \colon P\subsetneq R\} = \bigcup_{k=1}^{\infty}\Stop^k(R).
\end{equation*}
\par\medskip
This way, renaming $P''$ as $P$ in \refeq{eq:143}, we have
\begin{equation*}
\begin{aligned}
\ND_1 &\lesssim \sum_{R\in\Top} \theta_\mu(B_R)\sum_{P\in\Top\colon P\subsetneq R} A_0^{\frac{J(R_{P})-J(P)}{2}}\theta_\mu(B_{P})\mu(P) \\
	&= \sum_{R\in\Top}\theta_\mu(B_R)\sum_{k=1}^{\infty}\sum_{P\in\Stop^k(R)}A_0^{\frac{J(R_P)-J(P)}{2}}\theta_\mu(B_P)\mu(P) \\
	&\lesssim \sum_{R\in\Top}\theta_\mu(B_R)\sum_{k=1}^{\infty}A_0^{-\frac{k}{2}}\sum_{P\in\Stop^k(R)}\theta_\mu(B_P)\mu(P)^{\frac{1}{2}}\mu(P)^{\frac{1}{2}},
\end{aligned}
\end{equation*}
because $P\in\Stop^k(R) \Rightarrow J(P)-J(R_P)\geq k-1$. Then, using Cauchy-Schwarz's inequality twice, we get
\begin{equation*}
\begin{aligned}
\ND_1 &\lesssim  \sum_{R\in\Top}\theta_\mu(B_R)\sum_{k=1}^{\infty}A_0^{-\frac{k}{2}} \left(\sum_{P\in\Stop^{k}(R)}\theta_\mu(B_P)^2\mu(P)\right)^{\frac{1}{2}}\left(\sum_{P\in\Stop^k(R)}\mu(P)\right)^{\frac{1}{2}}\\
	&=\sum_{k=1}^{\infty}A_0^{-\frac{k}{2}}\sum_{R\in\Top}\theta_\mu(B_R)\mu(R)^{\frac{1}{2}}\left(\sum_{Q\in\Stop^k(R)}\theta_\mu(B_P)^2\mu(P)\right)^{\frac{1}{2}}\\
	&\leq \sum_{k=1}^{\infty}A_0^{-\frac{k}{2}}\left(\sum_{R\in\Top}\theta_\mu(B_R)^2\mu(R)\right)^{\frac{1}{2}}\left(\sum_{R\in\Top}\sum_{P\in\Stop^k(R)}\theta_\mu(B_P)^2\mu(P)\right)^{\frac{1}{2}}\\
	&\lesssim \sum_{R\in\Top}\theta_\mu(B_R)^2\mu(R),
\end{aligned}
\end{equation*}
as desired. 
\end{proof}
\subsection{The estimate of $\ND_2$.}

\begin{lemma}
\label{Le:AcotacionND2}
\begin{equation*}
\ND_2 \lesssim \sum_{R\in\Top} \theta_\mu(B_R)^2\mu(R).
\end{equation*}
\end{lemma}

\begin{proof}
Recall that 
\begin{equation*}
\ND_2 = \sum_{R\in\Top}\sum_{P\in\Stop(R)}\sum_{i=J(P)}^{\infty}\sum_{S\in\D_{m(J(P),i)}}\langle \chi_S T_i(\chi_{\R^d\setminus S}\mu),K_R\mu\rangle_\mu.
\end{equation*}

Fix $R\in\Top$, $P\in\Stop(R)$, $i\geq J(P)$ and $S\in\D_{m(J(P),i)}$. We have
\begin{equation*}
\langle \chi_ST_i(\chi_{\R^d\setminus S}\mu),K_R\mu\rangle_\mu = \int_S T_i(\chi_{\R^d\setminus S}\mu)K_R\mu d\mu.
\end{equation*}
Now, if $x\in S$,
\begin{equation*}
T_i(\chi_{\R^d\setminus S}\mu)(x) = \int_{\R^d\setminus S}\varphi_i(x-y)k(x,y)d\mu(y) = \int_{y\not\in S,\, 0.001A_0^{-i-1}<|x-y|<0.01A_0^{i}}\varphi_i(x-y)k(x,y)d\mu(y),
\end{equation*}
so $T_i(\chi_{\R^d\setminus S}\mu)(x) = 0$ unless $\dist(x,E\setminus S)<0.01A_0^{-i}$. Thus, if we denote
\begin{equation*}
\partial_i S = \{x\in S\colon \dist(x,E\setminus S)\leq 0.01A_0^{-i}\}
\end{equation*}
we have that
\begin{equation*}
\supp(\chi_{S}T_i(\chi_{\R^d\setminus S}\mu))\subset\partial_i S.
\end{equation*}
Then,
\begin{equation*}
\langle \chi_ST_i(\chi_{\R^d\setminus S\mu}),K_R\mu\rangle_\mu = \int_{\partial_i S} T_i(\chi_{\R^d\setminus S}\mu)K_R\mu d\mu = \sum_{M\in\D_i\colon M\subset S}\int_{\partial_i S\cap M} T_i(\chi_{\R^d\setminus S}\mu)K_R\mu d\mu.
\end{equation*}

Now, for $M\in\D_i$ with $M\subset S$ and $x\in\partial_i S\cap M$, we have
\begin{equation*}
\begin{aligned}
|T_i(\chi_{\R^d\setminus S}\mu)(x)| &= \left| \int_{y\not\in S,\,0.001A_0^{-i-1}<|x-y|<0.01A_0^{i}}\varphi_i(x-y)k(x,y)d\mu(y)\right|\\
	&\lesssim \int_{0.001A_0^{-i-1}<|x-y|<0.01A_0^{i}}\frac{d\mu(y)}{|x-y|^n} \leq \frac{\mu[B(x,0.01A_0^{-i})]}{A_0^{-ni}}\lesssim \theta_\mu[1.01B_M].
\end{aligned}
\end{equation*}
Therefore,
\begin{equation*}
\begin{aligned}
|\langle \chi_ST_i(\chi_{\R^d\setminus S\mu}),K_R\mu\rangle_\mu| &\leq \sum_{M\in\D_i\colon M\subset S}\left|\int_{\partial_i S\cap M} T_i(\chi_{\R^d\setminus S}\mu)K_R\mu d\mu \right|\\
	&\lesssim \sum_{M\in\D_i\colon M\subset S}\theta_\mu[1.01B_M]\int_{\partial_i S\cap M}|K_R\mu|d\mu.
\end{aligned}
\end{equation*}

Then, if we denote 
\begin{equation*}
\partial_i\D_{m(J(P),i)}=\bigcup_{S\in\D_m(J(P),i)}\partial_i S
\end{equation*}
we have
\begin{equation*}
\begin{aligned}
\left|\sum_{i=J(P)}^{\infty}\sum_{S\in\D_{m(J(P),i)}}\langle \chi_ST_i(\chi_{\R^d\setminus S}\mu),K_R\mu\rangle_\mu\right| &\lesssim \sum_{i=J(P)}^{\infty}\sum_{S\in\D_{m(J(P),i)}}\sum_{M\in\D_i\colon M\subset S}\theta_\mu[1.01B_M]\int_{\partial_i S\cap M}|K_R\mu|d\mu \\
	&\lesssim \sum_{i=J(P)}^{\infty}\sum_{M\in\D_i\colon M\subset P}\theta_\mu[1.01B_M]\int_{\partial_{J(M)}\D_{m(J(P),J(M))}\cap M}|K_R\mu|d\mu\\
	&=\sum_{P'\in \Top\colon P'\subset P}\sum_{M\in\Tree(P')}\theta_\mu[1.01B_M]\int_{\partial_{J(M)}\D_{m(J(P),J(M))}\cap M}|K_R\mu|d\mu.
\end{aligned}
\end{equation*}

Here we have that $\theta_\mu[1.01B_M]\lesssim \theta_\mu(B_{P'})$ for $M\in\Tree(P')$, and therefore
\begin{equation*}
\begin{aligned}
\left|\sum_{i=J(P)}^{\infty}\sum_{S\in\D_{m(J(P),i)}}\langle \chi_ST_i(\chi_{\R^d\setminus S}\mu),K_R\mu\rangle_\mu\right| &\lesssim \sum_{P'\in\Top\colon P'\subset P}\theta_\mu(B_{P'})\sum_{M\in\Tree(P')}\int_{\partial_{J(M)}\D_{m(J(P),J(M))}\cap M}|K_R\mu|d\mu\\
	&\lesssim \sum_{P'\in\Top\colon P'\subset P}\theta_\mu(B_{P'})\sum_{i=J(P')}^{\infty}\int_{\partial_{i}\D_{m(J(P),i)}\cap P'}|K_R\mu|d\mu.
\end{aligned}
\end{equation*}
\par\medskip
Here we use Cauchy-Schwarz's inequality to get
\begin{equation*}
\left|\sum_{i=J(P)}^{\infty}\sum_{S\in\D_{m(J(P),i)}}\langle \chi_ST_i(\chi_{\R^d\setminus S}\mu),K_R\mu\rangle_\mu\right| \lesssim \sum_{P'\in\Top\colon P'\subset P}\theta_\mu(B_{P'})\sum_{i=J(P')}^{\infty} ||K_R\mu||_{L^2(\chi_{P'}\mu)}\mu[(\partial_i\D_{m(J(P),i)})\cap P']^{\frac{1}{2}}
\end{equation*}

Now, given $R,P'\in\Top$ with $P'\subsetneq R$, we set
\begin{equation*}
\mu_{R,P'}=\left(\sum_{i=J(P')}^{\infty}\mu[(\partial_i\D_{m(J(R_{P'}),i)})\cap P']^{\frac{1}{2}}\right)^{2}
\end{equation*}
so that
\begin{equation*}
\begin{aligned}
\ND_2 &\lesssim \sum_{R\in\Top}\sum_{P'\in\Top\colon P'\subsetneq R}\theta_\mu(B_{P'})||K_R\mu||_{L^2(\chi_{P'}\mu)}\mu_{R,P'}^{\frac{1}{2}}\\
	&=\sum_{k=1}^{\infty}\sum_{R\in\Top}\sum_{Q\in\Stop^k(R)}\theta_\mu(B_Q)||K_R\mu||_{L^2(\chi_Q\mu)}\mu_{R,Q}^{\frac{1}{2}},
\end{aligned}
\end{equation*}
and here, we use Cauchy-Schwarz's inequality twice again to get
\begin{equation*}
\begin{aligned}
\ND_2 &\lesssim \sum_{k=1}^{\infty}\sum_{R\in\Top} \left(\sum_{Q\in\Stop^k(R)}||K_R\mu||^2_{L^2(\chi_Q\mu)}\right)^{\frac{1}{2}}\left(\sum_{Q\in\Stop^k(R)}\theta_\mu(B_Q)^2\mu_{R,Q}\right)^{\frac{1}{2}}\\
	&\leq \sum_{k=1}^{\infty}\sum_{R\in\Top}||K_R\mu||_{L^2(\mu)}\left(\sum_{Q\in\Stop^k(R)}\theta_\mu(B_Q)^2\mu_{R,Q}\right)^{\frac{1}{2}}\\
	&\leq \sum_{k=1}^{\infty}\left(\sum_{R\in\Top}||K_R\mu||^2_{L^2(\mu)}\right)^\frac{1}{2}\left(\sum_{R\in\Top}\sum_{Q\in\Stop^k(R)}\theta_\mu(B_Q)^2\mu_{R,Q}\right)^{\frac{1}{2}}\\
	&\lesssim\left(\sum_{R\in\Top}\theta_\mu(B_R)^2\mu(R)\right)^\frac{1}{2} \sum_{k=1}^{\infty}\left(\sum_{R\in\Top}\sum_{Q\in\Stop^k(R)}\theta_\mu(B_Q)^2\mu_{R,Q}\right)^{\frac{1}{2}},
\end{aligned}
\end{equation*}
where the last inequality follows from lemma \ref{Le:Diagonal}. Therefore, if we prove that
\begin{equation*}
\sum_{k=1}^{\infty}\left(\sum_{R\in\Top}\sum_{Q\in\Stop^k(R)}\theta_\mu(B_Q)^2\mu_{R,Q}\right)^{\frac{1}{2}} \lesssim \left(\sum_{R\in\Top}\theta_\mu(B_R)^2\mu(B_R)\right)^{\frac{1}{2}},
\end{equation*}
we will reach the desired conclusion. To do so, recall that for fixed $k\geq 1$, $R\in\Top$ and $Q\in\Stop^k(R)$
\begin{equation*}
\mu_{R,Q}=\left(\sum_{i=J(Q)}^{\infty}\mu[(\partial_i\D_{m(J(R_{Q}),i)})\cap Q]^{\frac{1}{2}}\right)^{2}.
\end{equation*}

Now, for $i\geq J(Q)$,
\begin{equation*}
\begin{aligned}
\mu[(\partial_i\D_{m(J(R_{Q}),i)})\cap Q] &= \sum_{S\in\D_{m(J(R_Q),i)}\colon S\subset Q}\mu(\partial_i S) \\
	&=\sum_{S\in\D_{m(J(R_Q),i)}\colon S\subset Q}\mu\left(\left\{x\in S\colon \dist(x,\R^d\setminus S)<\frac{0.01A_0^{-i}}{\l(S)}\l(S)\right\}\right)\\
	&\lesssim \sum_{S\in\D_{m(J(R_Q),i)}\colon S\subset Q}\left(\frac{l_i}{l_m}\right)^{\frac{1}{2}}\mu(3.5B_S)\\
	&\lesssim A_0^{\frac{J(R_Q)-i}{2}}\mu(B_Q),
\end{aligned}
\end{equation*}
where the penultimate inequality follows from \refeq{eqfk490}. Therefore,
\begin{equation*}
\begin{aligned}
\mu_{R,Q}\lesssim \left(\sum_{i=J(Q)}^{\infty}\left(A_0^{\frac{J(R_Q)-i}{2}}\mu(B_Q)\right)^{\frac{1}{2}}\right)^2 = \mu(B_Q)\left(\sum_{i=J(Q)}^{\infty}A_0^{\frac{J(R_Q)-i}{4}}\right)^2\lesssim \mu(B_Q)A_0^{\frac{J(R_Q)-J(Q)}{2}}
\end{aligned}
\end{equation*}
and so
\begin{equation*}
\begin{aligned}
\sum_{k=1}^{\infty}\left(\sum_{R\in\Top}\sum_{Q\in\Stop^k(R)}\theta_\mu(B_Q)^2\mu_{R,Q}\right)^{\frac{1}{2}} &\lesssim \sum_{k=1}^{\infty}\left(\sum_{R\in \Top}\sum_{Q\in\Stop^k(R)}\theta_\mu(B_Q)^2\mu(B_Q)A_0^{\frac{J(R_Q)-J(Q)}{2}}\right)^{\frac{1}{2}}\\
	&\lesssim \sum_{k=1}^{\infty}A_0^{-\frac{k}{4}}\left(\sum_{R\in\Top}\sum_{Q\in\Stop^k(R)}\theta_\mu(B_Q)^2\mu(B_Q)\right)^\frac{1}{2}\\
	&\lesssim \left(\sum_{R\in\Top}\theta_\mu(B_R)^2\mu(B_R)\right)^{\frac{1}{2}}\\
	&\lesssim \left(\sum_{R\in\Top}\theta_\mu(B_R)^2\mu(R)\right)^{\frac{1}{2}},
\end{aligned}
\end{equation*}
as desired.
\end{proof}

\section{The proof of the main lemma \ref{Le:MainLemma}}

This is a straightforward consequence of lemmas \ref{Le:Diagonal}, \ref{Le:AcotacionND1}, \ref{Le:AcotacionND2} and \ref{Th:CoronaDesc}. Indeed, going back to section $6$,

\begin{equation*}
||T\mu||^2_{L^2(\mu)} = \sum_{R\in\Top}||K_R\mu||^2_{L^2(\mu)} + \sum_{R,R'\in\Top} \langle K_R\mu, K_{R'}\mu\rangle_\mu.
\end{equation*}

Now, by lemma \ref{Le:Diagonal},
\begin{equation*}
\sum_{R\in\Top}||K_R\mu||^2_{L^2(\mu)} \lesssim \sum_{R\in\Top}\theta_\mu(B_R)^2\mu(R),
\end{equation*}
and by lemmas \ref{Le:AcotacionND1} and \ref{Le:AcotacionND2}
\begin{equation*}
\left|\sum_{R,R'\in\Top} \langle K_R\mu, K_{R'}\mu\rangle_\mu\right| \lesssim \sum_{R\in\Top}\theta_\mu(B_R)^2\mu(R),
\end{equation*}
so
\begin{equation*}
||T\mu||^2_{L^2(\mu)} \lesssim \sum_{R\in\Top}\theta_\mu(B_R)^2\mu(R) \lesssim ||\mu|| + \iint_0^1 \beta_{\mu,2}(x,r)\theta_\mu[B(x,r)]\frac{dr}{r}d\mu(x),
\end{equation*}
as desired.

\section{The proof of Corollary \ref{Cor:LipHarmCap}}

The key idea behind the proof is to use Volberg's characterization of Lipschitz harmonic capacity \cite[Lemma 5.15]{Volberg}, which states that
\begin{equation*}
\kappa(E) \approx \sup \mu(E),
\end{equation*}
where the supremum is taken over all positive Borel measures $\mu$ supported on $E$ such that $\mu[B(x,r)]\leq r^n$ for all $x\in\R^{n+1}$ and all $r>0$ and such that the $n$-dimensional Riesz transform $\mathcal{R}$ with respect to $\mu$ is bounded in $L^2(\mu)$ with norm $\leq 1$. 
\par\medskip
Then, to prove Corollary \ref{Cor:LipHarmCap}, let $\mu$ be a positive Borel measure supported on $E$ satisfying \refeq{eq:SUPKAPPA}. Then, clearly $\mu[B(x,r)]\leq r^n$ for all $x\in\R^{n+1}$ and all $r>0$, and furthermore, applying Theorem \ref{Th:MainThm}, we get that $\mathcal{R}_\mu$ is bounded in $L^2(\mu)$ and its norm is bounded by some absolute constant. Therefore, for an appropriate multiple $\nu$ of $\mu$ we have that $\nu[B(x,r)]\leq r^n$ and $||\mathcal{R}_\nu||_{L^2(\nu)\rightarrow L^2(\nu)}\leq 1$, and so $\mu(E) \lesssim \nu(E) \lesssim \kappa(E)$, as desired.

\end{document}